\documentclass[10pt]{elsarticle}

\usepackage[%
paperheight = 24cm,%
paperwidth = 16.5cm,%
left = 1.5cm,%
right = 1.5cm,%
bottom = 2cm,%
top = 2.3cm,%
marginparwidth = 1.2cm,%
marginparsep = 1mm,%
]{geometry}

\journal{Stochastic Processes and their Applications}

\usepackage{microtype}

\newtheorem{theorem}{Theorem}[section]
\newtheorem{lemma}[theorem]{Lemma}
\newtheorem{proposition}[theorem]{Proposition}

\newdefinition{remark}[theorem]{Remark}

\newproof{proof}{Proof}

\newproof{proofof}{Proof of Theorem~\ref{thm:MA-clt}}

\usepackage{mathtools,etoolbox,enumitem}

\counterwithin{equation}{section}
\allowdisplaybreaks

\newcommand{\argmrk}{\,\cdot\,}
\newcommand{\HS}{\textup{HS}}

\DeclareMathOperator{\cov}{Cov}
\DeclareMathOperator{\dom}{dom}
\DeclareMathOperator{\tr}{Tr}

\usepackage{amssymb,dsfont}
\newcommand{\1}{\mathds{1}}
\newcommand{\bbN}{\mathbb{N}}
\newcommand{\bbP}{\mathbb{P}}
\newcommand{\bbR}{\mathbb{R}}
\newcommand{\bbZ}{\mathbb{Z}}

\newcommand{\cC}{\mathcal{C}}
\newcommand{\cF}{\mathcal{F}}

\newcommand{\cH}{\mathcal{H}}
\newcommand{\cL}{\mathcal{L}}
\newcommand{\cN}{\mathcal{N}}
\newcommand{\cS}{\mathcal{S}}

\newcommand{\com}{^{\mathsf{c}}}

\DeclarePairedDelimiterXPP\abs[1]{\mathop{}}\lvert\rvert{}{#1}%
\DeclarePairedDelimiterXPP\norm[1]{\mathop{}}{\lVert}{\rVert}{}{%
  \ifblank{#1}{\argmrk}{#1}%
}%
\DeclarePairedDelimiterXPP\iprod[1]{\mathop{}}\langle\rangle{}{%
  \ifblank{#1}{\argmrk, \argmrk}{#1}%
}%
\DeclarePairedDelimiter{\Set}{\{}{\}}%
\DeclarePairedDelimiterXPP\Expt[1]{\mathbb{E}}{[}{]}{}{#1}%

\DeclarePairedDelimiterX{\fp}[1]{\{}{\}}{
  \ifblank{#1}{\,\cdots\,}{#1}%
}

\DeclarePairedDelimiterX{\ip}[1]{[}{]}{%
  \ifblank{#1}{\,\cdots\,}{#1}%
}

\DeclarePairedDelimiterX{\floor}[1]{\lfloor}{\rfloor}{%
  \ifblank{#1}{\,\cdots\,}{#1}%
}

\DeclarePairedDelimiterX{\ceil}[1]{\lceil}{\rceil}{%
  \ifblank{#1}{\,\cdots\,}{#1}%
}

\DeclareMathOperator{\sgn}{sgn}

\newcommand{\isp}{\,}
\newcommand{\di}{\textup{d}}
\newcommand{\Di}{\isp\di}

\usepackage{hyperref}

\hypersetup{%
  hidelinks,%
}

\begin{document}
\fontsize{9pt}{12pt}\selectfont
\begin{frontmatter}

  \title{Multi-Dimensional Normal Approximation\\ of Heavy-Tailed Moving
    Averages}%

  \author[author1]{Ehsan Azmoodeh}%
  \ead{ehsan.azmoodeh@liverpool.ac.uk}%
  \address[author1]{Department of Mathematical Sciences, University of
    Liverpool, Mathematical Sciences Building, Liverpool, L69 7ZL,
    United Kingdom}

  \author[author2]{Mathias Mørck
    Ljungdahl\corref{mycorrespondingauthor}}%
  \ead{ljungdahl@math.au.dk}%
  \address[author2]{Department of Mathematics, Aarhus University, Ny
    Munkegade 118, DK-8000 Aarhus C, Denmark}%
  \cortext[mycorrespondingauthor]{Corresponding author}

  \author[author3]{Christoph Thäle}%
  \ead{christoph.thaele@rub.de}%
  \address[author3]{Faculty of Mathematics, Ruhr University Bochum,
    Universitätsstraße 150, 44801 Bochum, Germany}

  % \address[mymainaddress]{1600 John F Kennedy Boulevard, Philadelphia}
  % \address[mysecondaryaddress]{360 Park Avenue South, New York}

  \begin{abstract}
    In this paper we extend the refined second-order Poincaré
    inequality for Poisson functionals from a one-dimensional to a
    multi-dimensional setting. Its proof is based on a multivariate
    version of the Malliavin--Stein method for normal approximation on
    Poisson spaces. We also present an application to partial sums of
    vector-valued functionals of heavy-tailed moving averages. The
    extension allows a functional with multivariate arguments,
    i.e. multiple moving averages and also multivariate values of the
    functional. Such a set-up has previously not been explored in the
    framework of stable moving average processes. It can potentially
    capture probabilistic properties which cannot be described solely
    by the one-dimensional marginals, but instead require the joint
    distribution.
  \end{abstract}

  \begin{keyword}
    Central limit theorem\sep heavy-tailed moving average\sep Lévy
    process\sep Malliavin--Stein method\sep Poisson random measure\sep
    second-order Poincaré inequality

    \MSC[2010] 60F05\sep 60G10\sep 60G15\sep 60G52\sep 60G55\sep 60H07
  \end{keyword}
\end{frontmatter}

% \linenumbers

\section{Introduction}

\noindent In recent decades the combination of Malliavin calculus and
Stein's method for normal approximation has led to a plethora of
Gaussian limit theorems in fields ranging from stochastic geometry,
over cosmology to statistics. Classically, the assumptions require
third or fourth moment conditions which makes the Malliavin--Stein
method unsuitable for distributions with heavier tails. However, in
\cite{BassBerr} a careful differentiation between small and large
values has led to a refined so-called second-order Poincaré inequality
for Poisson functionals, which allows to circumvent these difficulties
to a certain extent. Based on the approach in \cite{PeccMult} the
principal goal of this paper is to obtain a multivariate extension of
the central results of \cite{BassBerr}. This opens the possibility to
capture properties of the underlying process not accessible solely by
the one-dimensional marginal distributions. As a side-result we also
generalize the weak convergence result from
\cite[Theorem~1.1]{BassBerr} to a non-casual setting and due to the
choice of metric for probability laws we additionally remove the
non-trivial requirement of a non-zero variance of the Gaussian limit.

We shall now define the heavy-tailed moving average model to which we
are going to apply our general multivariate central limit theorem. Let
$L = (L_t)_{t \in \bbR}$ be a two-sided Lévy process with no Gaussian
component and Lévy measure $\nu$. We assume that the latter admits a
Lebesgue density $w : \bbR \to \bbR$ such that
\begin{equation}
  \label{eq:stable-levy-measure}
  \abs{w(x)} \leq C \abs{x}^{- 1 - \beta}
\end{equation}
for all $x \neq 0$, some $\beta \in (0, 2)$ and a constant $C >
0$. Hence, the distribution of $L_1$ exhibits $\beta$-stable tails. Consider then for each
$i \in \Set{1, \ldots, m}$, $m \in \bbN$, the process
\begin{equation}
  \label{eq:MA}
  X_t^i \coloneqq \int_{\bbR} g_i(t - s) \Di L_s, \qquad t \in \bbR,
\end{equation}
for some measurable function $g_i : \bbR \to \bbR$. Necessary and
sufficient conditions for the integral to exists are given in
\cite{RajpSpec} and if $L$ is symmetric around zero, i.e. if $-L_1$
and $L_1$ are identically distributed, then we mention that a
sufficient condition is
$\int_{\bbR} \abs{g_i(s)}^{\beta} \Di s < \infty$.

The main examples of kernels $g_i$ we consider satisfy a power-law
behaviour around zero and at infinity. Henceforth we shall assume for
all $i \in \Set{1, \ldots, m}$ the existence of a constant $K > 0$
together with exponents $\alpha_i > 0$ and $\kappa_i \in \bbR$ such
that
\begin{equation}
  \label{eq:kernel-ass}
  \abs{g_i(x)} \leq K \bigl(\abs{x}^{\kappa_i} \1_{[0, a_i)}(\abs{x})
  + \abs{x}^{-\alpha_i} \1_{[a_i, \infty)}(\abs{x}) \bigr) 
\end{equation}
for all $x \in \bbR$, where $a_i>0$ are suitable splitting points,
which may alter the constant $K$. Without~loss of generality we choose
$a_i = 1$ for all $i \in \Set{1, \ldots, m}$ and let $K$ stand for the
corresponding constant. Note in particular that we do not assume that
$X^i$ at \eqref{eq:MA} is a casual moving average as is assumed in
\cite[Theorem~1.1, Equation~(1.6)]{BassBerr}.

The main objects of interest in this paper are rescaled and centred
partial sums of multi-dimensional functionals of the joint
distribution $X_s = (X_s^1, \ldots, X_s^m)$, namely
\begin{equation}
  \label{eq:main-statistic}
  V_n(X; f) = \frac{1}{\sqrt{n}} \sum_{s = 1}^{n} (f(X_s^1, \ldots, X_s^m) - \Expt{f(X_0^1, \ldots, X_0^m)}),
  \quad n \in \bbN,
\end{equation}
where $f : \bbR^m \to \bbR^d$ is a suitable Borel-measurable function,
with $d$ being some positive integer. Observe that $V_n(X; f)$ is a
$d$-dimensional random vector and for convenience we shall denote by
$V_n^i(X; f)$ its $i$th coordinate. We remark that in the one
dimensional case $d = m = 1$ the distributional convergence of
$V_n(X; f)$, as $n \to \infty$, is studied for general functions $f$
in \cite{BassOnLi} and here the so-called Appell rank of $f$ is seen
to play an important role. The results in that paper also imply that
one cannot in general expect convergence in distribution after
rescaling with the factor $\sqrt{n}$ as in \eqref{eq:main-statistic}
or a Gaussian limiting distribution if the memory of the processes are
too long, i.e. if the $\alpha_i$ are too close to $0$. We shall see
that if the tails are not too heavy and the memory is not too long,
which in our case means that $\alpha_i \beta > 2$, we do in fact have
convergence in distribution of $V_n(X; f)$ to a Gaussian random
variable and we shall discuss the \emph{speed} of this convergence by
considering an appropriate metric on the space of probability laws on
$\bbR^d$, see Section~\ref{sec:limit-theory} below. To~conclude such a
result, we could also in principle rely on a multivariate second-order
Poincaré inequality for random vectors of Poisson functionals in
\cite{SchuMult}. But as already observed in the one-dimensional case,
the existing bounds are not suitable for the application to Lévy
driven moving averages just described. In fact, in this specific
situation the bounds in \cite{SchuMult} do not even tend to zero, as
$n$ increases. Against this background, we will develop in this paper
a refined multivariate second-order Poincaré inequality for general
random vectors of Poisson functionals, which is more adapted to our
situation and allows us to distinguish carefully between small and
large values. We believe that this result is of independent interest
as well. This eventually paves the way to the central limit theory for
the random vectors $V_n(X; f)$.

One possible motivation for the extension of the theory from
\cite{BassBerr} to a multivariate set-up is the fact that important
properties of random processes, such as self-similarity, are
determined by the finite dimensional distributions of $X$ and not
solely by the one-dimensional marginals. The one-dimensional theory,
i.e. the case $m = d = 1$, could so-far capture only probabilistic
properties of the distribution of $X_1$. Indeed, as seen in
\cite[Example~2.3]{LjunMult} a joint three-dimensional distribution
$(X_1, X_2, X_3)$ is required to identify the self-similarity
parameter~$H$ of the linear fractional stable motion. Such a
requirement is fulfilled by Theorem~\ref{thm:MA-clt} according to
Remark~\ref{rem:MA-CLT}\ref{it:rem:MA-CLT:3} and \cite{LjunMult}
indeed uses Theorem~\ref{thm:MA-clt} as basis for a minimal contrast
estimator of multi-parameter heavy-tailed moving averages. This also
explains the shortcomings of \cite{LjunAMin}, where a ratio estimator
had to be used in conjunction with the minimal contrast approach.

We would like to mention finally that the case $m = 1$ and general $d$
has been considered in the seminal paper \cite{PipiCent}. Since here
$m$ is equal to $1$, the main result of that paper is not able to deal
with neither finite dimensional distributions such as
$(X_1, X_2, X_3)$ nor functionals whose arguments depends on multiple,
different, moving averages with the same driving Lévy process.

\section{Main results}
\label{sec:limit-theory}

\subsection{A refined multivariate second-order Poincaré inequality}
\label{sec:Poincare}

\noindent Consider a measurable space $(S, \cS)$ equipped with a
$\sigma$-finite measure $\mu$. Let $\eta$ be a Poisson process on
$(S, \cS)$ with intensity measure $\mu$. This means that $\eta$ is a
collection of random variables of the form $\eta(B)$, $B \in \cS$,
with the properties that
\begin{enumerate}[label = (\roman*)]
\item for each $B \in \cS$ with $\mu(B) < \infty$ the random variable
  $\eta(B)$ is Poisson distributed with mean $\mu(B)$,

\item for $m \in \bbN$ and pairwise disjoint
  $B_1, \ldots, B_m \in \cS$ with
  $\mu(B_1), \ldots, \mu(B_m) < \infty$ the random variables
  $\eta(B_1), \ldots, \eta(B_m)$ are independent.
\end{enumerate}
We can and will regard $\eta$ as a random function from an underlying
probability space $(\Omega, \cF, \bbP)$ to $\cN$, the space of all
integer-valued $\sigma$-finite measures on $(S, \cS)$. The set $\cN$
is equipped with the evaluation $\sigma$-algebra, i.e. the
$\sigma$-algebra generated by the evaluation mappings
$\mu \mapsto \mu(A)$, $A \in \cS$.

To each Poisson process $\eta$ we associate the Hilbert space
$\cL^2_\eta(\bbP)$ consisting of all square integrable Poisson
functionals $F$, i.e. those random variables for which there exists a
function $\phi : \cN \to \bbR$ such that almost surely
$F = \phi(\eta) \in \cL^2(\bbP)$. Finally, we introduce the notion of
the Malliavin derivative in a Poisson setting, which is also known as
the add-one-cost operator. For each $z \in S$ and
$F = \phi(\eta) \in \cL^2_\eta(\bbP)$ we define $D_z F$ as
\begin{equation*}
  D_z F \coloneqq \phi(\eta + \delta_z) - \phi(\eta),
\end{equation*}
and note that $D F$ is a bi-measurable map from $\Omega \times S$ to
$\bbR$. In a straightforward way this definition extends to
vector-valued Poisson functionals. Indeed, consider
$F = (F_1, \ldots, F_d)$ where each $F_i$ lies in $\cL^2_\eta(\bbP)$,
then the Malliavin derivative $D_z F$ at $z \in S$ is given by
\begin{equation*}
  D_z F = (D_z F_1, \ldots, D_z F_d).
\end{equation*}
Similarly to $D_zF$ we may introduce the iterated Malliavin derivative
$D^2 F$ of $F$ by putting
\begin{equation*}
  D^2_{z_1, z_2} F \coloneqq D_{z_1}(D_{z_2}F)
  = D_{z_2}(D_{z_1}F), \qquad z_1, z_2 \in S.
\end{equation*}
For further background material on Poisson processes we refer to the
treatments in \cite{LastLect,LastPois,PeccStoc}---for the Malliavin
formalism on Poisson spaces we refer to Section~\ref{subsec:Malliavin}
below.

To measure the distance between (the laws of) two random vectors $X$
and $Y$ taking values in $\bbR^d$ we use the so-called $d_3$-distance,
see \cite{PeccMult}.  To introduce it, assume that
$\Expt{\norm{X}^2_{\bbR^d}}, \Expt{\norm{Y}^2_{\bbR^d}} < \infty$,
where $\norm{}_{\bbR^d}$ stands for the Euclidean norm in
$\bbR^d$. The $d_3$-distance between $X$ and $Y$, denoted by
$d_3(X, Y)$, is given by
\begin{equation*}
  d_3 (X, Y) \coloneqq \sup_{\varphi \in \cH_3} \abs[\big]{\Expt{\varphi(X)} - \Expt{\varphi(Y)}},
\end{equation*}
where the class $\cH_3$ of test functions indicates the collection of
all thrice differentiable functions $\varphi : \bbR^d \to \bbR$
(i.e. $\varphi \in \cC^3(\bbR^d, \bbR)$) such that
$\norm{\varphi''}_{\infty} \leq 1$ and
$\norm{\varphi'''}_{\infty} \leq 1$, where
\begin{align*}
  \norm{\varphi''}_\infty &\coloneqq \max_{1 \leq i, j \leq d} \sup_{x \in \bbR^d} \abs[\Big]{\frac{\partial^2}{\partial x_i \partial x_j} \varphi(x)},
  \\
  \norm{\varphi'''}_\infty &\coloneqq \max_{1 \leq i, j, k \leq d} \sup_{x \in \bbR^d} \abs[\Big]{\frac{\partial^3}{\partial x_i \partial x_j\partial x_k} \varphi(x)}.
\end{align*}

We can now formulate our multivariate second-order Poincaré
inequality, which generalizes \cite[Theorem~3.1]{BassBerr} and refines
\cite[Theorem 1.1]{SchuMult}. Its proof, which is given in
Section~\ref{sec:ProofThm1} below, is based on the Malliavin--Stein
technique for normal approximation of random vectors of Poisson
functionals. For two Poisson functionals $F, G \in \cL^2_\eta(\bbP)$
we define the quantities
\begin{align*}
  \gamma_1^2(F, G) &\coloneqq 3 \int_{S^3} \Expt[\big]{(D^2_{z_1, z_3} F)^2 (D^2_{z_2, z_3} F)^2}^{1/2}
                     \Expt[\big]{(D_{z_1} G)^2 (D_{z_2} G)^2 }^{1/2}
                     \isp \mu^3(\di z_1, \di z_2, \di z_3),
  \\
  \gamma_2^2(F, G) &\coloneqq \int_{S^3} \Expt[\big]{(D^2_{z_1, z_3} F)^2 (D^2_{z_2, z_3} F)^2}^{1/2} \Expt[\big]{(D^2_{z_1, z_3} G)^2 (D^2_{z_2, z_3} G)^2}^{1/2} \isp \mu^3(\di z_1, \di z_2, \di z_3).
\end{align*}
Moreover, for $x, y \in \bbR$ we denote by
$x \wedge y = \min \Set{x, y}$ the minimum of $x$ and $y$.

\begin{theorem}
  \label{thm:wasserstein-dist}
  Let $d \geq 1$ and assume that $F_1, \ldots, F_d \in \cL^2_\eta(\bbP)$
  satisfy $DF_i \in \cL^2(\bbP \otimes \mu)$ and $\Expt{F_i} = 0$ for
  all $i \in \Set{1, \ldots, d}$. Let
  $\sigma_{i k} \coloneqq \Expt{F_i F_k}$ and define the covariance
  matrix $\Sigma^2 = (\sigma_{ik})_{i,k = 1}^{d}$. Let
  $Y \sim N_d(0, \Sigma^2)$ be a centred Gaussian random vector with
  covariance matrix $\Sigma^2$ and put $F \coloneqq (F_1, \ldots,
  F_d)$. Then
  \begin{equation*}
    d_3(F, Y) \leq \sum_{i, k = 1}^{d} (\gamma_{1}(F_i,F_k) + \gamma_{2}(F_i,F_k))
    + \gamma_3,
  \end{equation*}
  where the term $\gamma_3$ is defined as
  \begin{equation}
    \label{eq:gamma3}
    \gamma_3
    \coloneqq \sum_{i, j, k = 1}^{d}
    \int_S \Expt[\big]{\abs{D_z F_j D_z F_k}^{3/2} \wedge \norm{D_z F}_{\bbR^d}^{3/2}}^{2/3} \Expt[\big]{\abs{D_z F_i}^3}^{1/3} \isp \mu(\di z).
  \end{equation}
\end{theorem}

\begin{remark}
  \label{rem:CLT}
  \leavevmode
  \begin{enumerate}[label = (\roman*)]
  \item The difference between Theorem~\ref{thm:wasserstein-dist} and
    \cite[Theorem 1.1]{SchuMult} lies in the term $\gamma_3$. We
    emphasize that the bound in \cite{SchuMult} does not lead to a
    meaningful error bound in the application to heavy-tailed moving
    averages we consider in the next section as the corresponding
    $\gamma_3$-term in \cite{SchuMult} would diverge. Similarly to the
    univariate case, the bound provided by
    Theorem~\ref{thm:wasserstein-dist} is much more suitable for our
    purposes as it leads to a reasonable error bound, which tends to
    zero, as the number of observations $n$ there tends to infinity.

  \item It is in principal possible to derive error bounds as in
    Theorem~\ref{thm:wasserstein-dist} for probability metrics different
    from the $d_3$-metric. Namely, assuming in addition that the
    covariance matrix $\Sigma^2$ is \emph{positive definite}, one can
    deal with the $d_2$-distance used in \cite{PeccMult} and even with
    the convex distance introduced and studied in \cite{SchuMult}. Since
    the corresponding error bounds for these notions of distance become
    rather long and technical, we refrain from presenting results in
    this direction. Moreover, in our application in the next section it
    seems in general rather difficult to check whether or not the
    covariance matrix is positive definite. This is another reason for
    us considering only the $d_3$-distance.
    
  \item We would like to point out that quantitative central limit
    theorems for random vectors of Poisson functionals having a
    \textit{finite} Wiener--It\^o chaos expansion with respect to the
    $d_3$-distance were obtained \cite{LastMome}. Specifically, random
    vectors of so-called Poisson U-statistics were considered in
    \cite{LastMome} together with applications in stochastic geometry to
    Poisson process of $k$-dimensional flat in $\bbR^n$.
  \end{enumerate}
\end{remark}

\subsection{Asymptotic normality of multivariate heavy-tailed moving
  averages}

\noindent Here, we present our application of the refined multivariate
second-order Poincaré inequality formulated in the previous
section. For this recall the set-up described in the
introduction. Especially, recall the definition of the random
processes $(X_t^i)_{t \in \bbR}$, $i \in \Set{1, \ldots, m}$ from
\eqref{eq:MA}. Also recall that the exponents $\alpha_i$ control the
memory of the processes $X^i$. Given the limit theory for heavy-tailed
moving averages as developed in \cite{LjunALim} it comes as no
surprise that the smallest such $\alpha_i$ will be of dominating
importance. Hence, we define
\begin{equation*}
  \underline{\alpha} = \min \Set{\alpha_1, \ldots, \alpha_m}, \qquad \text{and similary} \qquad
  \overline{\alpha} = \max \Set{\alpha_1, \ldots, \alpha_m}.
\end{equation*}
Finally, by $\cC_b^2(\bbR^m, \bbR^d)$ we denote the space of bounded
functions $f : \bbR^m \to \bbR^d$ which are twice continuously
differentiable and have all partial derivatives up to order two
bounded by some constant.

\begin{theorem}
  \label{thm:MA-clt}
  Fix $d, m \geq 1$. Let $(X_t^i)$, $i = 1, \ldots, m$, be moving
  averages as in \eqref{eq:MA} with Lévy measure having density $w$
  satisfying \eqref{eq:stable-levy-measure} for some
  $\beta \in (0, 2)$ and kernels $g_i$ which satisfy
  \eqref{eq:kernel-ass} with $\alpha_i \beta> 2$ and
  $\kappa_i > -1/\beta$. Let a function
  $f = (f_1, \ldots, f_d) \in \cC_b^2(\bbR^m, \bbR^d)$ be given and
  consider $V_n(X; f)$ as in \eqref{eq:main-statistic} based on $f$
  and $X = (X^1, \ldots, X^m)$. Let
  $\Sigma_n = \cov(V_n(X; f))^{1 / 2}$ denote a positive semi-definite
  square root of the covariance matrix $\cov(V_n(X; f))$ of the
  $d$-dimensional random vector $V_n(X; f)$. Then
  $\Sigma_n \to \Sigma = (\Sigma_{i,j})_{i, j = 1}^d$, as
  $n \to \infty$, where, for $i,j \in \Set{1, \ldots, d}$,
  \begin{equation}
    \label{eq:asymp-cov}
    \begin{aligned}
      \Sigma^2_{i, j} &= \sum_{s = 0}^{\infty} \cov(f_i(X_s^1, \ldots, X_s^m), f_j(X_0^1, \ldots, X_0^m))
      \\
                      &\qquad+ \sum_{s = 1}^{\infty} \cov(f_i(X_0^1, \ldots, X_0^m), f_j(X_s^1, \ldots, X_s^m)).
    \end{aligned}
  \end{equation}
  Moreover, $V_n(X; f)$ converges in distribution, as $n \to \infty$,
  to a $d$-dimensional centred Gaussian random vector
  $Y \sim N_d(0, \Sigma^2)$ with covariance matrix $\Sigma^2$. More
  precisely, there exists a constant $C > 0$ which only depends on
  $\underline{\alpha}$, $\overline{\alpha}$, $\beta$ and the sup-norms
  of the partial derivatives of $f$, such that
  \begin{equation*}
    d_3(V_n(X; f), Y) \leq C d^4 m^4
    \begin{cases}
      n^{-1/2}, & \text{if $\underline{\alpha} \beta > 3$,}
      \\
      n^{-1/2} \log(n), & \text{if $\underline{\alpha} \beta = 3$,}
      \\
      n^{(2 - \underline{\alpha} \beta)/2}, & \text{if
        $2 < \underline{\alpha} \beta < 3$.}
    \end{cases}
  \end{equation*}
\end{theorem}

\begin{remark}
  \label{rem:MA-CLT}
  \leavevmode
  \begin{enumerate}[label = (\roman*)]
  \item\label{it:rem:MA-CLT:1} We remark that in the special case
    $d = m = 1$ the order for the $d_3$-distance provided by
    Theorem~\ref{thm:wasserstein-dist} is precisely the same as that
    for the Wasserstein distance in \cite{BassBerr}. Note however, that even in this case our result extends the one in \cite{BassBerr}, since since we handle the non-casual case as well.

  \item We note that
    the first-order limit theory for the non-scaled and
    non-centred statistics
    \begin{equation*}
      \frac{1}{n} \sum_{s = 1}^{n} f(X_s^1, \ldots, X_s^m) ,\quad n \in \bbN,
    \end{equation*}
    is well-known for bounded functionals $f : \bbR^m \to
    \bbR^d$. Indeed, by ergodicity of Lévy moving averages, see
    \cite{PassMixi}, the non-centred and non-scaled statistic
    converges almost surely to $\Expt{f(X^1_0, \ldots, X^m_0)}$ by
    Birkhoff's ergodic theorem. Against this light it is then natural
    to study (weak) convergence of the scaled and centred statistic
    $V_n(X; f)$ at \eqref{eq:main-statistic}. In the case that
    $\underline{\alpha} \beta < 2$ one may obtain a non-central and
    non-Gaussian weak limit theorem. Indeed, if
    $\underline{\alpha} \beta < 2$ by \cite[Theorem~2.2]{MazuEsti} one
    obtains a skewed stable random variable as a limit, which shows
    that one cannot expect the central limit theorem to hold for
    $\underline{\alpha} \beta < 2$. We refer to \cite{BassOnLi} for a
    discussion for more general functionals in the high frequency
    case.

  \item\label{it:rem:MA-CLT:2} Even for particular functions
    $f = (f_1, \ldots, f_d)$, such as trigonometric functions, it
    seems to be a rather demanding task to check whether the
    covariance matrix $\Sigma^2$ is positive definite or not. Note in
    this context that even in the one-dimensional case $d = m = 1$ the
    question of whether the asymptotic variance constant is strictly
    positive or not is generally difficult. This is the reason why we
    are working with the $d_3$-distance in this paper, since more
    refined probability metrics usually require positive definiteness
    of the covariance matrix, see Remark~\ref{rem:CLT}.

  \item\label{it:rem:MA-CLT:3} It is straightforward to modify the
    proof of Theorem~\ref{thm:MA-clt} to the situation where
    $X = (X_1, \ldots, X_m)$ for some fixed moving average
    $(X_t)_{t \in \bbR}$ as in \eqref{eq:MA} and where the kernel $g$
    satisfy
    \begin{equation*}
      \abs{g(x)} \leq K \bigl(\abs{x}^{\kappa} \1_{[0, a)}(\abs{x}) + \abs{x}^{-\alpha} \1_{[a, \infty)}(\abs{x}) \bigr)
    \end{equation*}
    for some constants $a, \alpha, K > 0$ and $\kappa \in \bbR$ such
    that $\alpha \beta > 2$ and $\kappa > - 1/\beta$. In
    this case the kernel of $X^i = X_i$ is simply
    $g_i = g(i + \argmrk)$. Choosing an appropriate functional $f$ in
    $V_n(X; f)$, such as the empirical characteristic function of $X$,
    opens up the possibility of inference on $(X_t)_{t \in \bbR}$ based
    on not only the marginal distribution $X_1$ as in much of the
    previous literature, but also on the joint distribution
    $(X_1, \ldots, X_m)$.
  \end{enumerate}
\end{remark}

\noindent As in \cite{BassBerr}, Theorem~\ref{thm:wasserstein-dist}
can be applied to particular processes $(X_t^i)$. We mention here the
linear fractional stable noises, which may be regarded as heavy-tailed
extensions of a fractional Gaussian noise. Let $L$ be a $\beta$-stable
Lévy process with $\beta \in (0,2)$ and put
\begin{equation*}
  X_t^i \coloneqq Y_t - Y_{t - 1} \qquad\text{for}\qquad
  Y_t^i \coloneqq \int_{-\infty}^t \bigl[(t - s)_+^{H_i - 1/\beta} - (-s)_+^{H_i - 1/\beta} \bigr] \Di L_s,
\end{equation*}
where $H_1, \ldots, H_m \in (0,1)$ (if $\beta = 1$ we additionally
suppose that $L$ is symmetric). In this case,
$\alpha_i = 1 - H_i + 1/\beta$ for all $i \in \Set{1, \ldots, m}$ and
the condition $\underline{\alpha} \beta > 2$ translates into
$\beta \in (1, 2)$ and
$\max \Set{H_1, \ldots, H_m} < 1 - \frac{1}{\beta}$. Note that since
$\beta > 1$ we automatically have that $\underline{\alpha} \beta <
3$. In this set-up the bound in Theorem~\ref{thm:wasserstein-dist}
reads as follows:
\begin{equation*}
  d_3(V_n(X; f), Y) \leq C \, n^{1/2 - \beta(1 - \max \Set{H_1, \ldots, H_m}) / 2}.
\end{equation*}
As a second application we mention a stable Ornstein--Uhlenbeck
process. Again, for a $\beta$-stable Lévy process $L$ with
$\beta \in (0, 2)$ define for $i \in \{1, \ldots, m\}$,
\begin{equation*}
  X_t^i \coloneqq \int_{-\infty}^{t} e^{-\lambda_i(t - s)}  \Di L_s,
\end{equation*}
where $\lambda_1, \ldots, \lambda_m > 0$. In this case, the parameters
$\alpha_1, \ldots, \alpha_m$ may be arbitrary and the error bound in
Theorem~\ref{thm:wasserstein-dist} reduces to
\begin{equation*}
  d_3(V_n(X; f), Y) \leq C \, n^{-1/2}.
\end{equation*}
In a similar spirit, one my consider multivariate quantitative central
limit theorems for functionals of linear fractional Lévy noises or
of stable fractional ARIMA processes, see \cite{BassBerr} for the
corresponding one-dimensional situations.

\section{Background material}

\subsection{Malliavin calculus on Poisson spaces}
\label{subsec:Malliavin}

\noindent To take advantage of the powerful Malliavin--Stein method we
need to recall some background material regarding the Malliavin
formalism on Poisson spaces. For further details we refer to
\cite{LastLect,LastPois,NualIntr}.

Throughout this section $\eta$ denotes a Poisson process with
intensity measure $\mu$ defined on some measurable space $(S, \cS)$
and over some probability space $(\Omega, \cF, \bbP)$. We start by
recalling that any $F \in \cL^2_\eta(\bbP)$ admits a chaos expansion
(with convergence in $\cL^2(\bbP)$). That is,
\begin{equation}
  \label{eq:chaos-exp}
  F = \sum_{n = 0}^{\infty} I_n(f_n),
\end{equation}
where $I_n$ denotes the $n$th order Wiener--It\^o integral with
respect to the compensated Poisson process $\eta - \mu$ and the
kernels $f_n \in \cL^2(\mu^n)$ are symmetric functions (i.e. they are
invariant under permutations of its variables). Especially,
$I_0(c) = c$ for all $c \in \bbR$.

The Kabanov--Skorohod integral $\delta$ is defined for a subclass of
random processes $u \in \cL^2(\bbP \otimes \mu)$ having chaotic
decomposition
\begin{equation*}
  u(z) = \sum_{n = 0}^{\infty} I_n(h_n(\argmrk, z)),
\end{equation*}
where for each $z\in S$ the function $h_n(\argmrk, z)$ is symmetric
and belongs to $\cL^2(\mu^n)$. Denoting by $\tilde{h}$ the canonical
symmetrization of a function $h : S^n \to \bbR$, i.e.
\begin{equation*}
  \tilde{h}(z_1, \ldots, z_n)
  = \frac{1}{n!} \sum_{\sigma \in S_n} h(z_{\sigma(1)}, \ldots, z_{\sigma(n)}),
\end{equation*}
with $S_n$ being the group of all permutations of
$\Set{1, \ldots, n}$, we put
\begin{equation*}
  \delta(u) \coloneqq \sum_{n = 0}^{\infty} I_{n + 1}(\tilde{h}_n),
\end{equation*}
whenever
$\sum_{n = 0}^{\infty} (n + 1)! \norm{\tilde{h}_n}_{\smash{\cL^2(\mu^{n
      + 1})}}^2 < \infty$ (we indicate this by writing
$u \in \dom \delta$), where $\norm{}_{\cL^2(\mu^{n + 1})}$ denotes the
usual $\cL^2$-norm with respect to $\mu^{n + 1}$.

Next, we shall define the two operators
$L : \dom L \to \cL^2_\eta(\bbP)$ and
$L^{-1} : \cL^2_\eta(\bbP) \to \cL^2_\eta(\bbP)$, where $\dom L$
denotes the class of Poisson functionals $F \in \cL^2_\eta(\bbP)$ with
chaos expansion as in \eqref{eq:chaos-exp} satisfying
$\sum_{n = 1}^{\infty} n^2 n!  \norm{f_n}_{\cL^2(\mu^n)}^2 <
\infty$. Then, we define
\begin{equation*}
  LF \coloneqq - \sum_{n = 1}^{\infty} n I_n(f_n).
\end{equation*}
Similarly, the pseudo-inverse $L^{-1}$ of $L$ acts on centred
$F \in \cL_\eta^2(\bbP)$ with chaotic expansion \eqref{eq:chaos-exp}
as follows:
\begin{equation*}
  L^{-1} F \coloneqq -\sum_{n = 1}^{\infty} \frac{1}{n} I_n(f_n).
\end{equation*}
Finally, we recall that for $F \in \cL_\eta^2(\bbP)$ with chaotic
expansion \eqref{eq:chaos-exp} satisfying
$\sum_{n = 0}^{\infty} (n + 1)! \norm{f_n}_{\cL^2(\mu^n)}^2 < \infty$
the Malliavin derivative admits the representation
\begin{equation*}
  D_z F = \sum_{n = 1}^{\infty} n I_{n - 1}(f_n(\argmrk, z)), \qquad z \in S.
\end{equation*}

Using these definitions and representations, one may prove the
following crucial formulas and relationships of Malliavin calculus,
which also play a prominent role in our approach:
\begin{enumerate}[label = (\arabic*)]
\item\label{it:op-rules:1} $L L^{-1} F = F$ if $F$ is centred.
  
\item\label{it:op-rules:2} $LF = -\delta DF$ for $F \in \dom L$.
  
\item\label{it:op-rules:3}
  $\Expt{F \delta(u)} = \Expt{\int_S (D_z F) u(z) \isp
    \mu(\di z)}$, when $u \in \dom \delta$.
\end{enumerate}

\subsection{Multivariate normal approximation by Stein's method}

\noindent Stein's method for multivariate normal approximation is a
powerful device to prove quantitative multivariate central limit
theorems. The proof of Theorem~\ref{thm:wasserstein-dist} is based on
the following result, which is known as Stein's~Lemma (see
\cite[Lemma~4.1.3]{NourNorm}). To present it, let us recall that the
Hilbert--Schmidt inner product between two $d \times d$ matrices
$A = (a_{i k})$ and $B = (b_{i k})$ is defined as
\begin{equation*}
  \iprod{A, B}_{\HS} = \tr(B^{\top} A)
  = \sum_{i, k = 1}^{d} b_{k i} a_{k i}.
\end{equation*}
Moreover, for a differentiable function $\varphi : \bbR^d \to \bbR$ we
shall write $\nabla \varphi$ for the gradient and $\nabla^2 \varphi$
for the Hessian of $\varphi$. Also, we let
$\iprod{}_{\bbR^d}$ denote the Euclidean scalar
product in $\bbR^d$.

\begin{lemma}[Stein's Lemma]
  \label{lem:stein}
  Let $\Sigma^2 \in \bbR^{d \times d}$ be a positive semi-definite
  matrix and $Y$ be a $d$-dimensional random vector. Then
  $Y \sim N_d(0, \Sigma^2)$ if and only if for all twice continuously
  differentiable functions $\varphi : \bbR^d \to \bbR$ with bounded
  derivatives one has that
  \begin{equation*}
    \Expt{\iprod{Y, \nabla \varphi(Y)}_{\bbR^d} - \iprod{\Sigma^2, \nabla^2 \varphi(Y)}_{\HS}} = 0.
  \end{equation*}
\end{lemma}

\section{Proof of Theorem \ref*{thm:wasserstein-dist}}
\label{sec:ProofThm1}

\noindent By definition of the $d_3$-distance we need to prove that
\begin{equation*}
  \abs{\Expt{\varphi(Y)} - \Expt{\varphi(F)}}
  \leq \sum_{i, k = 1}^{d} (\gamma_{1}(F_i,F_k) + \gamma_{2}(F_i,F_k)) + \gamma_3
\end{equation*}
for every function $\varphi \in \cH_3$. For this, we may assume that
$Y$ and $F$ are independent. We start out by applying the
interpolation technique already demonstrated in
\cite{PeccMult}. Consider the function $\Psi : [0, 1] \to \bbR$ given
by
\begin{equation*}
  \Psi(t) \coloneqq \Expt{\varphi(\sqrt{1 - t} F + \sqrt{t} Y)}, \qquad t \in [0, 1].
\end{equation*}
Note that from the mean value theorem it follows that
\begin{equation*}
  \abs{\Expt{\varphi(Y)} - \Expt{\varphi(F)}} = \abs{\Psi(1) - \Psi(0)}
  \leq \sup_{t \in (0, 1)} \abs{\Psi'(t)}.
\end{equation*}
Hence it is enough to consider $\Psi'$, which is given by
\begin{equation*}
  \Psi'(t) = \Expt[\big]{\iprod{\nabla \varphi(\sqrt{1 - t} F + \sqrt{t} Y), \tfrac{1}{2 \sqrt{t}} Y - \tfrac{1}{2 \sqrt{1 - t}} F}_{\bbR^d}}
  \eqqcolon \frac{1}{2 \sqrt{t}} T_1 - \frac{1}{2 \sqrt{1 - t}} T_2.
\end{equation*}
We consider the two terms $T_1$ and $T_2$ separately. For $T_1$ it
follows first by independence of $F$ and $Y$ and Stein's~Lemma (used
on the function $y \mapsto \varphi(\sqrt{1 - t} a + \sqrt{t} y)$ and
then dividing by $\sqrt{t}$) that
\begin{align*}
  T_1 &= \Expt{\iprod{\nabla \varphi(\sqrt{1 - t} F + \sqrt{t} Y), Y}_{\bbR^d}}
  \\
      &= \Expt[\big]{\Expt{\iprod{\nabla \varphi(\sqrt{1 - t} a + \sqrt{t} Y), Y}_{\bbR^d}} \mid_{a = F}}
  \\
      &= \sqrt{t}\, \Expt[\big]{\Expt{\iprod{\Sigma^2, \nabla^2 \varphi(\sqrt{1 - t} a + \sqrt{t} Y)}_{\HS}} \mid_{a = F}}.
\end{align*}
Let $\partial_i f$ denote the derivative of $f$ in the $i$th
coordinate. We have by independence of $F$ and $Y$ and the Malliavin
rules \ref{it:op-rules:1}--\ref{it:op-rules:3} rephrased at the end of
Section~\ref{subsec:Malliavin} that
\begin{align*}
  T_2 &= \Expt{\iprod{\nabla \varphi(\sqrt{1 - t} F + \sqrt{t} Y), F}_{\bbR^d}}
        = \sum_{i = 1}^{d} \Expt[\big]{\Expt{\partial_i \varphi(\sqrt{1 - t} F + \sqrt{t} a) F_i} \mid_{a = Y}}
  \\
      &= \sum_{i = 1}^{d} \Expt[\big]{\Expt{\partial_i \varphi(\sqrt{1 - t} F + \sqrt{t} a) L (L^{-1} F_i) }\mid_{a = Y}}
  \\
      &= -\sum_{i = 1}^{d} \Expt[\big]{\Expt{\partial_i \varphi(\sqrt{1 - t} F + \sqrt{t} a) \delta (D L^{-1} F_i)} \mid_{a = Y}}
  \\
      &= \sum_{i = 1}^{d} \Expt[\big]{\Expt{\iprod{D \partial_i \varphi(\sqrt{1 - t} F + \sqrt{t} a), -D L^{-1} (F_i)}_{\cL^2(\mu)} } \mid_{a = Y}}.
\end{align*}
Consider now the function $\varphi_i^{t, a} : \bbR^d \to \bbR$ defined by
\begin{equation*}
  \varphi_i^{t, a}(x) \coloneqq \partial_i \varphi(\sqrt{1 - t} x + \sqrt{t} a).
\end{equation*}
By Taylor expansion we can write
\begin{equation*}
  D_z \varphi_i^{t, a}(F)
  = \sum_{k = 1}^{d} \partial_k \varphi_i^{t, a}(F) (D_z F_k) + R_i^a(D_z F)
\end{equation*}
for any $z \in \bbR^d$, where the remainder term
$R_i^a(D_z F) = \sum_{j, k = 1}^{d} R_{i, j, k}^a(D_z F_k, D_z F_j)$
satisfies the estimate
\begin{equation}
  \label{eq:remainder-1}
  \begin{aligned}
    \abs{R_{i, j, k}^a(x, y)} &\leq \tfrac{1}{2} \abs{x y} \max_{k, l} \sup_{x \in \bbR^d} \abs[\big]{\partial_{k, l} \varphi_i^{t, a}(x)}
    \\
                              &\leq \tfrac{1}{2} \abs{x y} (1 - t) \max_{k, l} \sup_{x \in \bbR^d} \abs[\big]{\partial_{i, k, l} \varphi(\sqrt{1 - t} x + \sqrt{t} a)} 
    \\
                              &\leq \tfrac{1}{2} (1 - t) \abs{x y}.
  \end{aligned}
\end{equation}
Here we have used the definition of the class $\cH_3$. On the other
hand, the remainder term also satisfies the inequality
\begin{equation}
  \begin{aligned}
    \label{eq:remainder-2}
    \abs[\Big]{D_z \varphi_i^{t, a}(F) - \sum_{k = 1}^{d} \partial_k \varphi_i^{t, a}(F) (D_z F_k)}
    &\leq \abs{D_z \varphi_i^{t, a}(F)} + \abs{\iprod{\nabla \varphi_i^{t, a}(F), D_z F}_{\bbR^d}}
    \\
    &\leq 2 \norm{\nabla \varphi_i^{t, a}(F)}_{\bbR^d} \norm{D_z F}_{\bbR^d}
    \\
    &\leq 2 \sqrt{1 - t} \norm{D_z F}_{\bbR^d},
  \end{aligned}
\end{equation}
where we used again the mean value theorem and the Cauchy--Schwarz
inequality. We may thus rewrite $T_2$ as
\begin{align*}
  T_2 &= \sum_{i, k = 1}^{d} \Expt[\big]{\Expt{\iprod{\partial_k \varphi_i^{t, a}(F) (D F_k), - D L^{-1} (F_i)}_{\cL^2(\mu)}} \mid_{a = Y}}
  \\
      &\qquad + \sum_{i = 1}^{d} \Expt[\big]{\Expt{\iprod{R_i^a(D F), - D L^{-1}(F_i)}_{\cL^2(\mu)}} \mid_{a = Y}}
  \\
      &= \sqrt{1 - t} \sum_{i, k = 1}^{d} \Expt[\big]{ \partial_{k, i} \varphi(\sqrt{1 - t} F + \sqrt{t} Y) \iprod{ D F_k, - D L^{-1} (F_i)}_{\cL^2(\mu)}}
  \\
      &\qquad + \sum_{i = 1}^{d} \Expt[\big]{\Expt{\iprod{R_i^a(D F), - D L^{-1}(F_i)}_{\cL^2(\mu)} } \mid_{a = Y}}.
\end{align*}
From this together with the Cauchy--Schwarz inequality and the bounds
\eqref{eq:remainder-1} and \eqref{eq:remainder-2} it follows that
\begin{align*}
  \MoveEqLeft \abs{\Expt{\varphi(Y)} - \Expt{\varphi(F)}} \leq \sup_{t \in (0, 1)} \abs{\Psi'(t)}
  \\
  &\leq \sup_{t \in (0, 1)} \frac{1}{2} \sum_{i, k = 1}^{d} \Expt[\Big]{\abs[\big]{\partial_{i, k} \varphi(\sqrt{1 - t} F + \sqrt{t} X)} \abs[\big]{\sigma_{ik} - \iprod{D F_k, -D L^{-1} (F_i)}_{\cL^2(\mu)}}}
  \\
  &\qquad + \sup_{t \in (0, 1)} \frac{1}{2 \sqrt{1 - t}} \sum_{i = 1}^{d} \Expt[\big]{{\abs{\iprod{R_i^a(D F), - D L^{-1}(F_i)}_{\cL^2(\mu)}} } \mid_{a = Y}}
  \\
  &\leq \frac{1}{2} \sum_{i, k = 1}^{d} \Expt[\big]{\abs{\sigma_{ik} - \iprod{D F_k, - D L^{-1} F_i}_{\cL^2(\mu)}}}
  \\
  &\qquad + \sum_{i, j, k = 1}^{d} \int_S \Expt{(\abs{D_z F_j D_z F_k} \wedge \norm{D_z F}_{\bbR^d}) \abs{D_z L^{-1} F_i}} \isp \mu(\di z).
\end{align*}
Applying now Proposition~4.1 in \cite{LastNorm} to the first of these
terms yields the inequality
\begin{equation*}
  \sum_{i, k = 1}^{d}  \Expt[\Big]{\abs[\big]{\sigma_{i k} - \iprod{D F_k, - D L^{-1} F_i}_{\cL^2(\mu)}}}
  \leq 2 \sum_{i, k = 1}^{d} (\gamma_{1, i, k} + \gamma_{2, i, k}).
\end{equation*}
For the remainder term we deduce by Hölder's inequality with exponents
$3$ and $3 / 2$ that
\begin{align*}
  \MoveEqLeft \int_S \Expt[\big]{(\abs{D_z F_j D_z F_k} \wedge \norm{D_z F}_{\bbR^d}) \abs{D_z L^{-1} F_i}} \isp \mu(\di z)
  \\
  &\leq \int_S \Expt[\big]{(\abs{D_z F_j D_z F_k} \wedge \norm{D_z F}_{\bbR^d})^{3/2}}^{2/3} \Expt[\big]{\abs{D_z L^{-1} F_i}^3}^{1/3} \isp \mu(\di z)
  \\
  &\leq \int_S \Expt[\big]{\abs{D_z F_j D_z F_k}^{3/2} \wedge \norm{D_z F}_{\bbR^d}^{3/2}}^{2/3} \Expt[\big]{\abs{D_z F_i}^3}^{1/3} \isp \mu(\di z),
\end{align*}
where we also used the contraction inequality
$\Expt{\abs{D_z L^{-1} F_i}^p} \leq \Expt{\abs{D_z F_i}^p}$ from
\cite[Lemma~3.4]{LastNorm}, which is valid for all $p \geq 1$ and
$z \in \bbR^d$. This completes the proof of
Theorem~\ref{thm:wasserstein-dist}.\hfill$\Box$

\section{Proof of Theorem \ref*{thm:MA-clt}}

\noindent In order to apply Theorem~\ref{thm:wasserstein-dist} we need
to ensure first of all that the processes $(X_t^i)$ can be represented
in terms of a Poisson process. Indeed, following \cite{RosiRepr} and
\cite{BassBerr} we can represent $X^i$ as the integral
\begin{equation*}
  X_t^i = \int_{\bbR^2} g_i(t - s) x \bigl(\eta(\di s, \di x) - \tau(g_i(t - s) x) \Di s \isp \nu(\di x) \bigr)
  + \tilde{b}_i,
\end{equation*}
with
\begin{equation*}
  \tilde{b}_i \coloneqq \int_{\bbR} \Bigl( g_i(s) b + \int_{\bbR} (\tau(x g_i(s)) - g_i(s) \tau(x)) \isp \nu(\di x) \Bigr) \Di s,
\end{equation*}
and where $\eta$ is a Poisson process on $\bbR^2$ with intensity
measure $\mu(\di s, \di x) \coloneqq \di s \isp \nu(\di x)$. Here,
$\nu$ is the Lévy measure of $L$, $b$ the shift parameter in the
characteristic triple for $L_1$ and $\tau$ is a truncation function,
cf. (8.3)--(8.4) in \cite{SatoLevy}. This representation is also the
formal starting point for the proof of Theorem~\ref{thm:MA-clt}.

In what follows, $C$ will denote a strictly positive constant whose
value might change from occasion to occasion, but only depends
$\underline{\alpha}$, $\overline{\alpha}$, $\beta$ and the sup-norms
of the partial derivatives of the function $f$. If $C$ depends additionally on the parameter $m$ we shall write $C(m)$ to highlight this dependence.

\subsection{Estimating the Malliavin derivative}

\noindent We start out by deriving simple estimates on the Malliavin
derivative. By definition of the terms $\gamma_1$, $\gamma_2$,
$\gamma_3$ introduced in Section~\ref{sec:Poincare} it is sufficient
to consider the Malliavin derivatives of each of the coordinates of
$f = (f_1, \ldots, f_d)$ separately. So, let
$i \in \Set{1, \ldots, d}$ and $z_j = (x_j, t_j) \in \bbR^2$ for
$j \in \Set{1, 2}$ be given. Define for $z = (x, t) \in \bbR^2$ the
vector $\delta_s(z)$, with $s \in \bbR$, as
\begin{equation}
  \label{eq:delta-def}
  \delta_s(z) \coloneqq x(g_1(s - t), \ldots, g_m(s - t)) \in \bbR^m.
\end{equation}
The mean value theorem together with the Cauchy--Schwarz inequality
and the assumption that $f_i \in \cC_b^2(\bbR^m, \bbR)$ then yield the
existence of a constant $C > 0$ such that
\begin{equation}
  \label{eq:malliavin-ineq:1}
  \begin{aligned}
    \abs{D_{z_1} f_i(X_s^1, \ldots, X_s^m)}
    &= \abs{f_i((X_s^1, \ldots, X_s^m) + \delta_s(z_1)) - f_i(X_s^1, \ldots, X_s^m)}
    \\
    &\leq C (1 \wedge \norm{\delta_s(z_1)}_{\bbR^m}).
  \end{aligned}
\end{equation}
Similarly, we deduce again by the mean value theorem and boundedness
of $f_i$ and its derivatives the following inequality for the iterated
Malliavin derivative:
\begin{align}
  % \begin{aligned}
  \abs{D^2_{z_1, z_2} f_i(X_s^1, \ldots, X_s^m)}
  &= \bigl\lvert f_i((X_s^1, \ldots, X_s^m) + \delta_s(z_1) + \delta_s(z_2))
    \nonumber
  \\
  &\qquad - f_i((X_s^1, \ldots, X_s^m) + \delta_s(z_1))
    \label{eq:malliavin-ineq:2}
  \\
  &\qquad - f_i((X_s^1, \ldots, X_s^m) + \delta_s(z_2))
    + f_i(X_s^1, \ldots, X_s^m) \bigr\rvert
    \nonumber
  \\
  &\leq C (1 \wedge \norm{\delta_s(z_1)}_{\bbR^m}) (1 \wedge \norm{\delta_s(z_2)}_{\bbR^m}).
    \nonumber
    % \end{aligned}
\end{align}
Note that the estimates \eqref{eq:malliavin-ineq:1} and
\eqref{eq:malliavin-ineq:2} are purely deterministic and allow us to
replace stochastic terms by deterministic estimates of the underlying
kernels. This confirms in another context that many properties of
moving averages can be deduced solely from the driving spectral
density, see, for example, \cite{KonoSelf}.

\subsection{Analysing the asymptotic covariance matrix}

\noindent Define for each $k \in \bbZ$ and
$i, j \in \Set{1, \ldots, m}$ the integral
\begin{equation}
  \label{eq:rho-and-mu}
  \rho_{i, j, k} \coloneqq \int_{\bbR} \abs{g_i(x) g_j(x + k)}^{\beta/2} \Di x.
\end{equation}
Now, $\rho_{i, j, k}$ is closely related to the asymptotic
covariances, which motivates the following technical lemma, which in
turn leads to our assumption that $\alpha_i \beta > 2$ for any
$i \in \Set{1, \ldots, m}$.

\begin{lemma}
  \label{lem:cov-ineq}
  There is a constant $C > 0$ which only depends $\overline{\alpha}$
  and $\beta$ such that for all $k \in \bbZ$ with $\abs{k} \geq 2$ and
  any $i, j \in \Set{1, \ldots, m}$,
  \begin{equation*}
    \rho_{i, j, k} \leq C \abs{k}^{-(\alpha_i \wedge \alpha_j) \beta / 2}.
  \end{equation*}
\end{lemma}

\begin{proof}
  We split the integral in \eqref{eq:rho-and-mu} into the regions
  $(-1, 1)$ and $(-1, 1)\com \coloneqq \bbR \setminus (-1, 1)$. For
  the first region we note that for $x \in (-1, 1)$ one has that
  $x + k \notin (-1, 1)$, since we assumed that $\abs{k} \geq 2$. So
  by \eqref{eq:kernel-ass} and the substitution $u = x \abs{k}^{-1}$,
  we have
  \begin{align*}
    \int_{-1}^{1} \abs{g_i(x) g_j(x + k)}^{\beta / 2} \Di x
    &\leq C \int_{-1}^{1} \abs{x}^{\beta \kappa_i / 2} \abs{x + k}^{-\alpha_j \beta / 2} \Di x
    \\
    &= C \abs{k}^{\beta \kappa_i / 2 - \alpha_j \beta / 2 + 1} \int_{-\frac{1}{\abs{k}}}^{\frac{1}{\abs{k}}}
      \abs{u}^{\beta \kappa_i / 2} \abs{u + \tfrac{k}{\abs{k}}}^{-\alpha_j \beta / 2} \Di u.
  \end{align*}
  The second factor in the integral is bounded since the reverse
  triangle inequality shows that for $e \in \Set{-1, +1}$ and
  $u \in (-\abs{k}^{-1}, \abs{k}^{-1})$,
  \begin{equation*}
    \abs{u + e} \geq \abs{\abs{u} - \abs{-e}}
    = 1 - \abs{u}
    \geq 1 - \abs{k}^{-1}
    \geq \tfrac{1}{2}.
  \end{equation*}
  Hence,
  \begin{equation}
    \label{eq:rho-1}
    \begin{aligned}
      \int_{-1}^{1} \abs{g_i(x) g_j(x + k)}^{\beta / 2} \Di x
      &\leq  C \abs{k}^{\beta \kappa_i / 2 - \alpha_j \beta / 2 + 1}
        \int_{0}^{ \frac{1}{\abs{k}}} u^{\beta \kappa_i / 2} \Di u
      = C \abs{k}^{-\alpha_j \beta / 2},
    \end{aligned}
  \end{equation}
  where we used that $\kappa > -1/\beta$.
  %$\kappa_i > - 1/\beta > - 1/(2 \beta)$.
  The
  second integral region, $(-1, 1)\com$, is further split into the
  regions $(-1, 1)\com \cap (-k - 1, -k + 1) = (-k - 1, -k + 1)$ and
  $(-1, 1)\com \cap (-k - 1, -k + 1)\com$. For the first we obtain,
  after a translation, a term which is almost identical to the first
  integration region:
  \begin{equation}
    \label{eq:rho-2}
    \begin{aligned}
      \int_{-k - 1}^{ -k + 1} \abs{g_i(x) g_j(x + k)}^{\beta / 2} \Di u
      &\leq C \int_{-k - 1}^{ -k + 1} \abs{x}^{-\alpha_i \beta / 2} \abs{x + k}^{\beta \kappa_j / 2} \Di x
      \\
      &= C \int_{-1}^{1} \abs{x - k}^{-\alpha_i \beta / 2} \abs{x}^{\beta \kappa_j /2} \Di x
      \\
      &\leq C \abs{k}^{-\alpha_i \beta /2}.
    \end{aligned}
  \end{equation}
  The last term is more intricate. We start by writing
  \begin{align*}
    \MoveEqLeft \int_{(-1, 1)\com \cap (-k - 1, -k + 1)\com} \abs{g_i(x) g_j(x + k)}^{\beta / 2} \Di x
    \\
    &\leq C \int_{(-1, 1)\com \cap (-k - 1, -k + 1)\com} \abs{x}^{-\alpha_i\beta /2} \abs{x + k}^{-\alpha_j\beta / 2} \Di x
    \\
    &= C \abs{k}^{1 - (\alpha_i + \alpha_j) \beta / 2}
      \int_{(-\frac{1}{\abs{k}}, \frac{1}{\abs{k}})\com \cap (\frac{-k - 1}{\abs{k}}, \frac{-k + 1}{\abs{k}})\com}
      \abs{u}^{-\alpha_i \beta / 2} \abs{u + \tfrac{k}{\abs{k}}}^{-\alpha_j \beta / 2} \Di u.
  \end{align*}
  This shows two competing effects, but since they do not occur
  simultaneously we can isolate each by splitting further into the
  sub-regions $(-1 / 2, 1 / 2)\com$ and $(-1 / 2, 1 / 2)$, respectively. This yields
  \begin{equation}
    \label{eq:rho-3}
    \begin{aligned}
      \MoveEqLeft \int_{(-\frac{1}{\abs{k}}, \frac{1}{\abs{k}})\com \cap (\frac{-k - 1}{\abs{k}}, \frac{-k + 1}{\abs{k}})\com} \abs{u}^{-\alpha_i \beta / 2} \abs{u + \tfrac{k}{\abs{k}}}^{-\alpha_j \beta / 2} \Di u
      \\
      &= \int_{(\frac{-k - 1}{\abs{k}}, \frac{-k + 1}{\abs{k}})\com \cap (-\frac{1}{2}, \frac{1}{2})\com} \abs{u}^{-\alpha_i \beta / 2} \abs{u + \tfrac{k}{\abs{k}}}^{-\alpha_j \beta / 2} \Di u
      \\
      &\quad + \int_{(-\frac{1}{\abs{k}}, \frac{1}{\abs{k}})\com \cap (-\frac{1}{2}, \frac{1}{2})} \abs{u}^{-\alpha_i \beta / 2} \abs{u + \tfrac{k}{\abs{k}}}^{-\alpha_j \beta / 2} \Di u
      \\
      &\leq C \int_{(\frac{-k - 1}{\abs{k}}, \frac{-k + 1}{\abs{k}})\com}
        \abs{u + \tfrac{k}{\abs{k}}}^{-\alpha_j \beta / 2} \Di u
      \\
      &\quad+ C \int_{(-\frac{1}{\abs{k}}, \frac{1}{\abs{k}})\com} \abs{u}^{-\alpha_i \beta / 2} \Di u
      \\
      &\leq C \int_{\frac{1}{\abs{k}}}^{\infty} u^{-(\alpha_i \wedge \alpha_j) \beta / 2} \Di u
      \\
      &= C \abs{k}^{(\alpha_i \wedge \alpha_j) \beta/2 - 1},
    \end{aligned}
  \end{equation}
  where to obtain the first inequality we used again the reverse
  triangle inequality. Combining \eqref{eq:rho-1}--\eqref{eq:rho-3}
  completes the proof.\qed
\end{proof}

\begin{proposition}
  \label{prop:cov-conv}
  The series defining $\Sigma_{i, j}^2$ in \eqref{eq:asymp-cov} is
  absolutely convergent and we have that $\Sigma_n^2 \to \Sigma^2$, as
  $n \to \infty$. In particular, $\Sigma_n \to \Sigma$.
\end{proposition}

\begin{proof}
  First, we prove that the series in \eqref{eq:asymp-cov} converges
  absolutely. By symmetry it is enough to show that
  \begin{equation*}
    \sum_{s = 1}^{\infty} \abs{\cov(f_i(X_s^1, \ldots, X_s^m), f_j(X_0^1, \ldots, X_0^m))} < \infty
    \quad \text{for all $i, j \in \Set{1, \ldots, d}$.}
  \end{equation*}
  To this end, we recall that for two general functionals
  $F, G \in \cL_\eta^2(\bbP)$ of a Poisson process $\eta$ in a
  measurable space $(S, \cS)$ with intensity measure $\mu$ one has the
  inequality
  \begin{equation*}
      \cov(F,G) \leq \int_S \Expt[\big]{(D_zF)^2}^{1 / 2} \Expt[\big]{(D_z G)^2}^{1 / 2} \isp \mu(\di z).
  \end{equation*}
  In fact, this readily follows by applying two times the classical
  Poincaré inequality \cite[Equation~(1.8)]{LastPois} to the Poisson
  functionals $F + G$ and $F - G$, subtracting both relations from
  each other and finally applying the Cauchy--Schwartz inequality. In
  our situation this yields, together with the triangle inequality,
  \begin{align*}
    \MoveEqLeft\abs{\cov(f_i(X_s^1, \ldots, X_s^m), f_j(X_0^1, \ldots, X_0^m))}
    \\
    &\leq \int_{\bbR^2} \Expt[\big]{\abs{D_z f_i(X_s^1, \ldots, X_s^m)}^2 }^{1/2} \Expt[\big]{ \abs{D_z f_j(X_0^1, \ldots, X_0^m)}^2 }^{1/2}  \isp \mu(\di z)
    \\
    &\leq C \int_{\bbR} \Bigl( \int_{\bbR} (1 \wedge \abs{x}^2 \norm{(g_{\ell}(s - t))_{\ell = 1}^{m}}_{\bbR^m}
      \norm{(g_{\ell}(-t))_{\ell = 1}^{m}}_{\bbR^m}) \abs{x}^{-1 - \beta}  \Di x \Bigr) \Di t
    \\
    &= C \int_{\bbR} \norm{(g_\ell(s - t))_{\ell = 1}^{m}}_{\bbR^m}^{\beta/2}
      \norm{(g_\ell(- t))_{\ell = 1}^{m}}_{\bbR^m}^{\beta/2} \Di t
    \\
    &\leq C \sum_{k, \ell = 1}^{m} \int_{\bbR} \abs{g_\ell(s - t) g_k( - t)}^{\beta/2} \Di t
    \\
    &= C \sum_{k, \ell = 1}^{m} \rho_{k, \ell, s}
    \\
    & \leq C(m) s^{-\underline{\alpha} \beta / 2},
  \end{align*}
  where the first equality follows by the splitting the integral with
  respect to $x$ into the regions
  $(0, (\norm{(g_{\ell}(s - t))_{\ell = 1}^{m}}_{\bbR^m}
  \norm{(g_{\ell}(-t))_{\ell = 1}^{m}}_{\bbR^m})^{-1 / 2}]$ and
  $((\norm{(g_{\ell}(s - t))_{\ell = 1}^{m}}_{\bbR^m}
  \norm{(g_{\ell}(-t))_{\ell = 1}^{m}}_{\bbR^m})^{-1 / 2}, \infty)$,
  and the last inequality follows from Lemma~\ref{lem:cov-ineq}. Since
  $\underline{\alpha} \beta > 2$ by assumption the series in
  \eqref{eq:asymp-cov} converges absolutely. To deduce the convergence
  $\Sigma_n^2 \to \Sigma^2$ we use the stationarity of the sequence
  $(X_t^1, \ldots, X_t^m)$, $t \in \bbR$, to see that for any
  $i, j \in \Set{1, \ldots, d}$,
  \begin{align*}
    &\cov(V_n^i(X; f), V_n^j(X; f))
    \\
    &= n^{-1} \sum_{s, t = 1}^{n} \cov(f_i(X_s^1, \ldots, X_s^m), f_j(X_t^1, \ldots, X_t^m))
    \\
    &= n^{-1} \sum_{\substack{s, t = 1 \\ s \geq t}}^{n} \cov(f_i(X_{s - t}^1, \ldots, X_{s - t}^m), f_j(X_0^1, \ldots, X_0^m))
    \\
    &\qquad + n^{-1} \sum_{\substack{s, t = 1 \\ s < t}}^{n} \cov(f_i(X_0^1, \ldots, X_0^m), f_j(X_{t - s}^1, \ldots, X_{t - s}^m))
    \\
    &= \sum_{k = 0}^{n - 1} (1 - \tfrac{k}{n}) \cov(f_i(X_k^1, \ldots, X_k^m), f_j(X_0^1, \ldots, X_0^m))
    \\
    &\qquad + \sum_{k = 1}^{n - 1} (1 - \tfrac{k}{n}) \cov(f_i(X_0^1, \ldots, X_0^m), f_j(X_k^1, \ldots, X_k^m))
      \longrightarrow \Sigma_{i, j}^2,
  \end{align*}
  as $n \to \infty$, where the convergence follows by Lebesgue's
  dominated convergence theorem together with the absolute convergence
  of the series defining the limit $\Sigma_{i, j}^2$. Finally, the
  last claim simply follows by continuity of the square root.\qed
\end{proof}

\subsection{Bounding $d_3(V_n, Y)$}

\noindent Recall for $i, k \in \Set{1, \ldots, d}$ the definition of
the quantities $\gamma_{1}(F_i, F_k)$ and $\gamma_{2}(F_i, F_k)$ from
Section~\ref{sec:Poincare}, which are applied with $F_i = V^i_n(X; f)$
and $F_k = V^k_n(X;f)$. According to
Theorem~\ref{thm:wasserstein-dist} we have that for any $n \in \bbN$,
\begin{equation*}
  d_3(V_n(X; f), Y)
  \leq \sum_{i, k = 1}^{d} (\gamma_{1}(F_i,F_k) + \gamma_{2}(F_i,F_k))
  + \gamma_3,
\end{equation*}
where $\gamma_3$ is defined at \eqref{eq:gamma3}. We consider each of
these terms separately in the following three lemmas. Let us point to
the fact that the sum will converge at a speed of order $1/\sqrt{n}$,
whereas the $\gamma_3$-term will generally converge at a lower speed,
depending on the parameters $\underline{\alpha}$ and $\beta$. It is
also this last term that requires the stronger assumption
\eqref{eq:kernel-ass} rather than just
$\sum_{u = 0}^{\infty} \rho_{i, j, u} < \infty$ for all
$i, j \in \Set{1, \ldots, m}$. Indeed, as a product, in $\gamma_3$ we
carefully have to distinguish between small and large values, where
the latter are non-negligible for heavy-tailed moving averages.

\begin{lemma}
  \label{lem:gamma1}
  There exists a constant $C > 0$ such that
  $\gamma_1(F_i,F_k) \leq C m^3 n^{-1/2}$ for any
  $i, k \in \Set{1, \ldots, d}$.
\end{lemma}

\begin{proof}
  To simplify the notation put $V_n^i \coloneqq V_n^i(X; f)$ and
  recall that
  \begin{align*}
    \gamma_{1}^2(F_i,F_k)&= 3 \int_{(\bbR^2)^3} \Expt[\big]{(D_{z_1, z_3}^2 V_n^i)^2 (D_{z_2, z_3}^2 V_n^i)^2}^{1/2}
    \\
                         &\hspace{2cm}\times\Expt[\big]{(D_{z_1} V_n^k)^2 (D_{z_2} V_n^k)^2}^{1/2} \isp \mu^3(\di z_1, \di z_2, \di z_3).
  \end{align*}
  If $z_i = (x_i, t_i) \in \bbR^2$ for $i \in \Set{1, 2, 3}$, the
  integrand can be bounded using \eqref{eq:malliavin-ineq:1} and
  \eqref{eq:malliavin-ineq:2} as follows:
  \begin{align*}
    \MoveEqLeft[0] \Expt[\big]{(D_{z_1, z_3}^2 V_n^i)^2 (D_{z_2, z_3}^2 V_n^i)^2}^{1/2} \Expt[\big]{(D_{z_1} V_n^k)^2 (D_{z_2} V_n^k)^2}^{1/2}
    \\
    &\leq \frac{C}{n^2} \Bigl(\sum_{s_1 = 1}^{n} (1 \wedge \norm{\delta_{s_1}(z_1)}_{\bbR^m}) (1 \wedge \norm{\delta_{s_1}(z_3)}_{\bbR^m}) \Bigr)
    \\
    &\hspace{1.5cm}\times \Bigl(\sum_{s_2 = 1}^{n} (1 \wedge \norm{\delta_{s_2}(z_2)}_{\bbR^m}) (1 \wedge \norm{\delta_{s_2}(z_3)}_{\bbR^m}) \Bigr)
    \\
    &\hspace{1.5cm} \times \Bigl(\sum_{s_3 = 1}^{n} (1 \wedge \norm{\delta_{s_3}(z_1)}_{\bbR^m}) \Bigr)
      \Bigl(\sum_{s_4 = 1}^{n} (1 \wedge \norm{\delta_{s_4}(z_2)}_{\bbR^m})  \Bigr)
    \\
    &\leq \frac{C}{n^2} \sum_{s_1, \ldots, s_4 = 1}^{n} \bigl[ (1 \wedge (\norm{\delta_{s_1}(z_1)}_{\bbR^m} \norm{\delta_{s_3}(z_1)}_{\bbR^m})) (1 \wedge (\norm{\delta_{s_2}(z_2)}_{\bbR^m} \norm{\delta_{s_4}(z_2)}_{\bbR^m}))
    \\
    &\hspace{1.5cm} \times (1 \wedge (\norm{\delta_{s_1}(z_3)}_{\bbR^m} \norm{\delta_{s_2}(z_3)}_{\bbR^m})) \bigr]
    \\
    &\leq \frac{C}{n^2} \sum_{s_1, \ldots, s_4 = 1}^{n} \sum_{j_1, \ldots, j_6 = 1}^{m} (1 \wedge (x_1^2 \abs{g_{j_1}(s_1 - t_1) g_{j_2}(s_3 - t_1)}))
    \\
    &\hspace{1.5cm} \times (1 \wedge (x_2^2 \abs{g_{j_3}(s_2 - t_2) g_{j_4}(s_4 - t_2)})) (1 \wedge (x_3^2 \abs{g_{j_5}(s_1 - t_3) g_{j_6}(s_2 - t_3)})).
  \end{align*}
  Using the substitution $u_i = x_i^2 y_i$ for $y_i > 0$,
  $i \in \Set{1, 2, 3}$, one easily verifies the relation
  \begin{equation*}
    \int_{\bbR^3} (1 \wedge (x_1^2 y_1)) (1 \wedge (x_2^2 y_2)) (1 \wedge (x_3^2 y_3)) \abs{x_1 x_2 x_3}^{-1 - \beta} \Di x_1 \Di x_2 \Di x_3 = C y_1^{\beta / 2} y_2^{\beta / 2} y_3^{\beta / 2},
  \end{equation*}
  for $\beta \in (0, 2)$. This yields the bound
  \begin{align*}
    \gamma_{1}^2(F_i,F_k) &\leq \frac{C}{n^2} \sum_{j_1, \ldots, j_6 = 1}^{m} \sum_{s_1, \ldots, s_4 = 1}^{n} \Bigl( \int_{\bbR} \abs{g_{j_1}(s_1 - t_1) g_{j_2}(s_3 - t_1)}^{\beta / 2} \Di t_1
    \\
                          &\hspace{1.5cm} \times \int_{\bbR} \abs{g_{j_3}(s_2 - t_2) g_{j_4}(s_4 - t_2)}^{\beta / 2} \Di t_2
    \\
                          &\hspace{1.5cm} \times \int_{\bbR} \abs{g_{j_5}(s_1 - t_3) g_{j_6}(s_2 - t_3)}^{\beta / 2} \Di t_3 \Bigr)
    \\
                          &= \frac{C}{n^2} \sum_{j_1, \ldots, j_6 = 1}^{m} \sum_{s_1, \ldots, s_4 = 1}^{n} \rho_{j_1, j_2, s_3 - s_1} \rho_{j_3, j_4, s_4 - s_2} \rho_{j_5, j_6, s_2 - s_1}
    \\
                          &\leq \frac{C}{n} \sum_{j_1, \ldots, j_6 = 1}^{m} \Bigl(\sum_{u = -n}^{n} \rho_{j_1, j_2, u} \Bigr) \Bigl(\sum_{u = -n}^{n} \rho_{j_3, j_4, u} \Bigr) \Bigl(\sum_{u = -n}^{n} \rho_{j_5, j_6, u} \Bigr)
                            \leq \frac{C m^6}{n},
  \end{align*}
  where the penultimate inequality follows by substitution and the
  last inequality is due to Lemma~\ref{lem:cov-ineq}, and where we
  used that $\sum_{u = 0}^{\infty} \rho_{j, \ell, u} < \infty$ for all
  $j, \ell \in \Set{1, \ldots, m}$.\qed
\end{proof}

\begin{lemma}
  \label{lem:gamma2}
  There exists a constant $C > 0$ such that
  $\gamma_2(F_i, F_k) \leq C m^4 n^{-1 / 2}$ for all
  $i, k \in \Set{1, \ldots, d}$ and $n \in \bbN$.
\end{lemma}

\begin{proof}
  Using \eqref{eq:malliavin-ineq:2} we conclude that the integrand in
  the definition of $\gamma_2(F_i, F_k)$ is bounded as follows:
  \begin{align*}
    \MoveEqLeft \Expt[\big]{(D_{z_1, z_3}^2 V_n^i)(D_{z_2, z_3}^2 V_n^i)}^{1/2} \Expt[\big]{(D_{z_1, z_3}^2 V_n^k)(D_{z_2, z_3}^2 V_n^k)}^{1/2}
    \\
    &\leq \frac{C}{n^2} \sum_{s_1, \ldots, s_4 = 1}^{n} (1 \wedge (\norm{\delta_{s_1}(z_1)}_{\bbR^m} \norm{\delta_{s_3}(z_1)}_{\bbR^m})) (1 \wedge (\norm{\delta_{s_2}(z_2)}_{\bbR^m} \norm{\delta_{s_4}(z_2)}_{\bbR^m}))
    \\
    &\hspace{3cm} \times (1 \wedge (\norm{\delta_{s_1}(z_3)}_{\bbR^m} \norm{\delta_{s_2}(z_3)}_{\bbR^m} \norm{\delta_{s_3}(z_3)}_{\bbR^m} \norm{\delta_{s_4}(z_3)}_{\bbR^m}))
    \\
    &\leq \frac{C}{n^2} \sum_{s_1, \ldots, s_4 = 1}^{n} \sum_{j_1, \ldots, j_8 = 1}^{m} (1 \wedge (x_1^2 \abs{g_{j_1}(s_1 - t_1) g_{j_2}(s_3 - t_1)}))
    \\
    &\hspace{3cm} \times (1 \wedge (x_2^2 \abs{g_{j_3}(s_2 - t_2) g_{j_4}(s_4 - t_2)}))
    \\
    &\hspace{3cm} \times (1 \wedge (x_3^4 \abs{g_{j_5}(s_1 - t_3) g_{j_6}(s_2 - t_3) g_{j_7}(s_3 - t_3) g_{j_8}(s_4 - t_3)})).
  \end{align*}
  Moreover, as in the proof of the previous lemma we have that
  \begin{equation*}
    \int_{\bbR^3} (1 \wedge (x_1^2 y_1)) (1 \wedge (x_2^2 y_2)) (1 \wedge (x_3^4 y_3)) \abs{x_1 x_2 x_3}^{-1 - \beta} \Di x_1 \Di x_2 \Di x_3
    = C y_1^{\beta / 2} y_2^{\beta / 2} y_3^{\beta / 4}
  \end{equation*}
  for $\beta \in (0, 2)$ and real numbers $y_1, y_2, y_3 > 0$. This
  implies that
  \begin{align*}
    \gamma_{2}^2(F_i,F_k) &\leq \frac{C}{n^2} \sum_{j_1, \ldots, j_8 = 1}^{m} \sum_{s_1, \ldots, s_4 = 1}^{n} \int_{\bbR} \abs{g_{j_1}(s_1 - t_1) g_{j_2}(s_3 - t_1)}^{\beta / 2} \Di t_1
    \\
                          &\hspace{1cm} \times \int_{\bbR} \abs{g_{j_3}(s_2 - t_2) g_{j_4}(s_4 - t_2)}^{\beta / 2} \Di t_2
    \\
                          &\hspace{1cm} \times \int_{\bbR} \abs{g_{j_5}(s_1 - t_3) g_{j_6}(s_2 - t_3) g_{j_7}(s_3 - t_3) g_{j_8}(s_4 - t_3)}^{\beta / 4} \Di t_3
    \\
                          &\leq \frac{C}{n^2} \sum_{j_1, \ldots, j_8 = 1}^{m} \sum_{s_1, \ldots, s_4 = 1}^{n} \rho_{j_1, j_2, s_3 - s_1} \rho_{j_3, j_4, s_4 - s_2} (\rho_{j_5, j_6, s_2 - s_1} + \rho_{j_7, j_8, s_4 - s_3})
    \\
                          &\leq \frac{C m^8}{n},
  \end{align*}
  where the last inequality follows as in Lemma~\ref{lem:gamma1} and
  the penultimate inequality follows immediately from the fact that
  $\abs{x y} \leq x^2 + y^2$ for all $x, y \in \bbR$.\qed
\end{proof}

Finally, we consider the crucial term $\gamma_3$. The proof is a
non-casual adaptation of \cite[Lemma~4.5]{BassBerr}, which requires
the analysis of additional terms compared to the argument in
\cite{BassBerr}.

\begin{lemma}
  \label{lem:gamma3}
  There exists a constant $C > 0$ such that for all $n \in \bbN$,
  \begin{equation*}
    \gamma_3 \leq C d^4m^4
    \begin{cases}
      n^{-1/2}, & \text{if $\underline{\alpha} \beta > 3$,}
      \\
      n^{-1/2} \log(n), & \text{if $\underline{\alpha} \beta = 3$,}
      \\
      n^{(2 - \underline{\alpha} \beta) / 2}, & \text{if
        $2 < \underline{\alpha} \beta < 3$.}
    \end{cases}
  \end{equation*}
\end{lemma}

\begin{proof}
	Recall the definition of
	$\delta_s(z) = (\delta_s^1(z), \ldots, \delta_s^m(z))$ from
	\eqref{eq:delta-def} and define for $i\in\{1,\ldots,m\}$,
	\begin{equation*}
		A_n^i(z) \coloneqq \frac{1}{\sqrt{n}} \sum_{s = 1}^{n} (1 \wedge \abs{\delta_s^i(z)}).
	\end{equation*}
Denote by $f = (f_1, \ldots, f_d)$ the coordinate functions of $f$.
We start by writing
\begin{align*}
  \gamma_3
  &= \sum_{i, j, k = 1}^{d}
    \int_{\bbR^2} \Expt[\big]{\abs{D_z V_j D_z V_k}^{\frac{3}{2}} \wedge \norm{D_z V}_{\bbR^d}^{\frac{3}{2}}}^{\frac{2}{3}} \Expt[\big]{\abs{D_z V_i}^3}^{\frac{1}{3}} \isp \mu(\di z)
  \\
  &= \sum_{i, j, k = 1}^{d}
    \int_{\bbR^2} \Expt[\bigg]{\abs[\Big]{\Bigl(\frac{1}{\sqrt{n}} \sum_{s = 1}^{n} D_z f_j(X_s^1, \ldots, X_s^m)\Bigr)
    \Bigl(\frac{1}{\sqrt{n}} \sum_{s = 1}^{n} D_z f_k(X_s^1, \ldots, X_s^m)\Bigr)}^{\frac{3}{2}}
    \wedge \norm{D_z V}_{\bbR^d}^{\frac{3}{2}}}^{\frac{2}{3}}
  \\
  &\qquad\qquad\qquad \times\Expt[\bigg]{\abs[\Big]{\Bigl(\frac{1}{\sqrt{n}} \sum_{s = 1}^{n} D_z f_i(X_s^1, \ldots, X_s^m)\Bigr)}^3}^{1/3} \isp \mu(\di z).
\end{align*}
Using the triangle inequality, \eqref{eq:malliavin-ineq:1} and the fact that the $l^1$-norm dominates
the $l^2$-norm we find that
\begin{align*}
	\gamma_3
	&\leq C d^3
	\int_{\bbR^2} \Expt[\bigg]{\Bigl(\frac{1}{\sqrt{n}} \sum_{s = 1}^{n} (1 \wedge \norm{\delta_s(z)}_{\bbR^m})\Bigr)^2 \wedge \norm{D_z V}_{\bbR^d}^{3/2}}^{2/3}
	\Bigl(\frac{1}{\sqrt{n}} \sum_{s = 1}^{n} (1 \wedge \norm{\delta_s(z)}_{\bbR^m})\Bigr)  \isp \mu(\di z)
	\\
	&\leq  C d^3
	\int_{\bbR^2} \Expt[\bigg]{\Bigl(\frac{1}{\sqrt{n}} \sum_{s = 1}^{n} (1 \wedge \norm{\delta_s(z)}_{\bbR^m})\Bigr)^3
		\wedge \Bigl(\sum_{i = 1}^{d} \abs[\Big]{\frac{1}{\sqrt{n}}\sum_{s = 1}^{n} D_z f_i(X_s^1, \ldots, X_s^m)}\Bigr)^{3/2}}^{2/3}
	\\
	&\qquad\qquad \times\Bigl(\frac{1}{\sqrt{n}} \sum_{s = 1}^{n} (1 \wedge \norm{\delta_s(z)}_{\bbR^m})\Bigr)  \isp \mu(\di z)
	\\
	&\leq C d^3 \int_{\bbR^2} \biggl[ \Bigl(\frac{1}{\sqrt{n}} \sum_{s = 1}^{n} (1 \wedge \norm{\delta_s(z)}_{\bbR^m})\Bigr)^2
	\wedge \Bigl(\frac{d}{\sqrt{n}}\sum_{s = 1}^{n} (1 \wedge \norm{\delta_s(z)}_{\bbR^m})\Bigr) \biggr]
	\\
	&\qquad\qquad \times\Bigl(\frac{1}{\sqrt{n}} \sum_{s = 1}^{n} (1 \wedge \norm{\delta_s(z)}_{\bbR^m})\Bigr)  \isp \mu(\di z)
	\\
	&= C d^3 \int_{\bbR^2} \Bigl(\frac{1}{\sqrt{n}} \sum_{s = 1}^{n} (1 \wedge \norm{\delta_s(z)}_{\bbR^m})\Bigr)^3
	\wedge d \Bigl(\frac{1}{\sqrt{n}}\sum_{s = 1}^{n} (1 \wedge \norm{\delta_s(z)}_{\bbR^m})\Bigr)^2 \isp \mu(\di z).
\end{align*}
Together with the definition of $A_n^i$, the subadditivity of minimum
and Jensen's inequality we conclude that
\begin{align*}
	\gamma_3
	&\leq C d^4 \int_{\bbR^2} \Bigl(\sum_{i = 1}^{m} A_n^i(z)\Bigr)^3
	\wedge \Bigl(\sum_{j = 1}^{m} A_n^j(z)\Bigr)^2 \isp \mu(\di z)
	\\
	&\leq  C d^4 \int_{\bbR^2} \Bigl( m^2 \sum_{i = 1}^{m} A_n^i(z)^3 \Bigr)
	\wedge \Bigl(m\sum_{j = 1}^{m} A_n^j(z)^2 \Bigr) \isp \mu(\di z)
	\\
	&\leq C d^4 m^2 \sum_{i, j = 1}^{m} \int_{\bbR^2} A_n^i(z)^3 \wedge A_n^j(z)^2 \isp \mu(\di z)
\end{align*}	
We shall now prove a result akin to Lemma~4.6 in
  \cite{BassBerr}. Namely, for  we for any
  $p \in [1, 2]$, $q > 2$
  and $j, k \in \Set{1, \ldots, m}$ we have that
  \begin{equation}
    \label{eq:A_n-ineq}
    \int_{\bbR^2} (A_n^j(z)^p \wedge A_n^k(z)^q) \isp \mu(\di z) \leq C
    \begin{cases}
      n^{1 - q/2}, & \text{if $\underline{\alpha} \beta > q$,}
      \\
      n^{1 - q/2} \log(n), & \text{if $\underline{\alpha} \beta = q$,}
      \\
      n^{(2 - \underline{\alpha} \beta)/2}, & \text{if $2 < \underline{\alpha} \beta < 3$.}
    \end{cases}
  \end{equation}
  We split the integral into three regions $\abs{x} \in (0, 1)$,
  $\abs{x} \in [1, n^{\underline{\alpha}}]$ and
  $\abs{x} \in (n^{\underline{\alpha}}, \infty)$. Since $A_n^i$ is an
  even function of $x$ this yields the following decomposition:
  \begin{align*}
    \MoveEqLeft \int_{\bbR^2} A_n^j(x, s)^p \wedge A_n^k(x, s)^q \Di s \isp \nu(\di x)
    \\
    &= 2 \int_{0}^{1} \int_{\bbR} A_n^j(x, s)^p \wedge A_n^k(x, s)^q \Di s \isp \nu(\di x)
    \\
    &\qquad + 2 \int_{1}^{n^{\underline{\alpha}}} \int_{\bbR} A_n^j(x, s)^p \wedge A_n^k(x, s)^q \Di s \isp \nu(\di x)
    \\
    &\qquad + 2 \int_{n^{\underline{\alpha}}}^{\infty} \int_{\bbR} A_n^j(x, s)^p \wedge A_n^k(x, s)^q \Di s \isp \nu(\di x)
      \eqqcolon I_1 + I_2 + I_3.
  \end{align*}
  Throughout we let $\ell \in \Set{j, k}$ denote the index of either
  $A_n^j$ or $A_n^k$, depending on the context. Next, we bound each
  $I$-term separately.

  \paragraph{Bounding $I_1$}

  Consider $x \in \bbR$ and $s \in [-n - 1, n + 1] \setminus
  \bbZ$. Then splitting according to whether or not $x\in(-1, 1)$ and
  using assumption~\eqref{eq:kernel-ass} on $g$ we see that
  \begin{align*}
    \sum_{i = 1}^{n} 1 \wedge \abs{x g_\ell(i - s)}
    &\leq \sum_{i = 1}^{n} 1 \wedge \abs{x g_\ell(i - s)} \1_{(-1, 1)}(i - s)
      + \sum_{i = 1}^{n} \abs{x g_\ell(i - s)} \1_{(-1, 1)\com}(i - s)
    \\
    &\leq \sum_{i = 1}^{n} 1 \wedge \abs{x g_\ell(i - s)} \1_{(-1, 1)}(i - s)
      + C x \sum_{k \in \bbZ \setminus \Set{0}} \abs{k}^{-\underline{\alpha}}
    \\
    &\eqqcolon h_1(s, x) + h_2(s, x),
  \end{align*}
  where the series is finite since $\underline{\alpha} > 1$, which in
  turn follows from the assumption $\underline{\alpha} > 2 / \beta$
  and $0 < \beta < 2$. We analyse the first term, $h_1$, further. To
  get a hand on the number of summands in $h_1$ we need to introduce
  some further notation and terminology. For a real number
  $s \in \bbR$ we write $s = \ip{s} + \fp{s}$ for its decomposition
  into its integer and fractional part. Formally, the integer part is
  defined as
  \begin{equation*}
    \ip{s} =
    \begin{cases}
      \floor{s}, & \text{if $s \geq 0$,}
      \\
      \ceil{s}, & \text{if $s < 0$,}
    \end{cases}
  \end{equation*}
  where $ \floor{s} \coloneqq \max \Set{k \in \bbZ \mid k \leq s}$ and
  $\ceil{s} \coloneqq \min \Set{k \in \bbZ \mid s \leq k}$.  The
  fractional part is then defined as
  \begin{equation*}
    \fp{s} = s - \ip{s}
    =
    \begin{cases}
      s - \floor{s} \in [0, 1), & \text{if $s \geq 0$,}
      \\
      s - \ceil{s} \in (-1, 0], & \text{if $s < 0$.}
    \end{cases}
  \end{equation*}
  Before proceeding we observe the identities:
  \begin{equation}
    \label{eq:ip-fp-identities}
    \floor{-s} = - \ceil{s}, \quad
    \ceil{-s} = - \floor{s}
    \quad \text{and} \quad
    \fp{-s} = - \fp{s},
  \end{equation}
  where the second identity follows from the first and the third from
  the definition and the two preceding identities. Consider 
  $s \in \bbR$ and $i \in \bbZ$ such that $i - s \in (-1, 1)$. The
  reverse triangle inequality yields that
  \begin{equation*}
    \abs[\big]{\abs{i - \ip{s}} - \abs{\fp{s}}} \leq \abs{i - s} < 1,
  \end{equation*}
  and since $\fp{s} \in (-1, 1)$ this shows that
  $i - \ip{s} \in \Set{-1, 0, 1}$. Hence for some
  $u \in \Set{-1, 0, 1}$ we have that $i - s = u - \fp{s}$. From
  \eqref{eq:kernel-ass} we then obtain the bound
  \begin{align*}
    h_1(s, x) &\leq \sum_{u = -1}^{1} 1 \wedge \abs{x g_\ell(u - \fp{s})} \1_{(-1, 1)}(u - \fp{s})
    \\
              &\leq \sum_{u = -1}^{1} 1 \wedge x \abs{u - \fp{s}}^{\kappa_\ell} \1_{(-1, 1)}(u - \fp{s}).
  \end{align*}
  Note that the indicator function is crucial here due the following observation.
  Consider $s \in \bbR$ and $u \in \Set{-1, 0, 1}$ such that
  $u - \fp{s} \in (-1, 1)$. If $s \geq 0$, then $u \geq 0$, since if
  $u < 0$ then $u - s = -1 - s < -1$ and therefore
  $u - \fp{s} \notin (-1, 1)$. In other words, $\sgn(s) = \sgn(u)$.

  These considerations lead us to
  \begin{align*}
    \int_{-n - 1}^{n + 1} h_1(x, s)^q \Di s
    &\leq C \sum_{u = -1}^{1} \int_{-n - 1}^{n + 1} (1 \wedge x \abs{u - \fp{s}}^{\kappa_\ell} \1_{(-1, 1)}(u - \fp{s}))^q \Di s
    \\
    &= C \sum_{u = 0}^{1} \int_{0}^{n + 1} (1 \wedge x \abs{u - \fp{s}}^{\kappa_\ell})^q \Di s
    \\
    &= C \sum_{u = 0}^{1} (n + 1) \int_{0}^{1} (1 \wedge x \abs{u - \fp{s}}^{\kappa_\ell})^q \Di s
    \\
    &\leq C n \int_{0}^{1} (1 \wedge x \abs{s}^{\kappa_\ell})^q \Di s
  \end{align*}
  where the first equality follows from the fact that if $s < 0$ then
  $\abs{u - \fp{s}} = \abs{-u - \fp{-s}}$ (see
  \eqref{eq:ip-fp-identities}) together with our signum
  observation. The second equality and the second inequality follow by
  substitution. If $\kappa < 0$ we can continue our previous stream of
  inequalities and get
  \begin{align*}
    \int_{-n - 1}^{n + 1} h_1(x, s)^q \Di s
    &\leq C n x^q \int_{x^{-1 / \kappa_\ell}}^1 s^{\kappa_\ell q} \Di s
      + C n \int_{0}^{x^{-1/\kappa_\ell}} 1 \Di s
    \\
    &\leq C n (x^q \abs{\log(x)} + x^{-1 / \kappa_\ell}),
  \end{align*}
  which follows by splitting the first integral into the cases:
  $\kappa_\ell q = -1$ and $\kappa_\ell q \neq -1$. For $\kappa_\ell \geq 0$ we get
  that
  \begin{equation*}
    \int_{-n - 1}^{n + 1} h_1(x, s)^q \Di s \leq C n x^q
  \end{equation*}
  and the same bounds hold for our second term as well:
  \begin{equation*}
    \int_{-n - 1}^{n + 1} h_2(x, s)^q \Di s \leq C n x^q.
  \end{equation*}
  Combining these inequalities together with
  \eqref{eq:stable-levy-measure} leads to
  \begin{align}
    \int_{0}^{1} \int_{[-n - 1, n + 1]} A_n^k(x, s)^q \Di s \isp \nu(\di x)
    &\leq C n^{1 - q /2} \int_{0}^{1} \bigl(x^q \abs{\log(x)} + x^{-1 / \kappa_\ell} \1_{\Set{\kappa_\ell < 0}} \bigr) x^{-1 - \beta} \Di x
      \nonumber
    \\
    &\leq C n^{1 - q/2},
      \label{eq:I_1-1}
  \end{align}
  where the last inequality follows from the assumption
  $\kappa_\ell > - 1 / \beta$.

  Consider now $s \notin [-n - 1, n + 1]$ and note that
  $i - s \notin (-1, 1)$ for any $i \in \Set{1, \ldots, n}$.
  Then, assumption~\eqref{eq:kernel-ass} yields that
  \begin{equation}
    \label{eq:geometric-bound}
    \sum_{i = 1}^{n} 1 \wedge \abs{x g_\ell(i - s)}
    \leq C x \sum_{i = 1}^{n} \abs{i - s}^{-\underline{\alpha}}
    \leq C x \abs{\abs{1 - s}^{1 - \underline{\alpha}} - \abs{n - s}^{1 - \underline{\alpha}}}.
  \end{equation}
  In the case $q (1 - \underline{\alpha}) < -1$ we simply remove the non-positive
  terms in \eqref{eq:geometric-bound} to obtain the bound
  \begin{equation}
    \label{eq:x-0-1-outside-B-1}
    \begin{aligned}
      \MoveEqLeft \int_{[-n - 1, n + 1]\com} \abs[\big]{\abs{1 - s}^{1 -
      \underline{\alpha}} - \abs{n - s}^{1 - \underline{\alpha}}}^q \Di s
      \\
      &\leq \int_{-\infty}^{-n - 1} \abs{1 - s}^{q(1 - \underline{\alpha})} \Di s
        + \int_{n + 1}^{\infty} \abs{n - s}^{q(1 - \underline{\alpha})} \Di s
      \\
      &\leq \int_{1}^{\infty} s^{q(1 - \underline{\alpha})} \Di s < \infty.
    \end{aligned}
  \end{equation}
  Suppose now that $(1 - \underline{\alpha}) q > -1$. Then, by substitution and using the fact that
  $n \geq 2$, we have that
  \begin{equation}
    \label{eq:x-0-1-outside-B-2}
    \begin{aligned}
      \MoveEqLeft \int_{[-n - 1, n + 1]\com} \abs[\big]{\abs{1 - s}^{1 - \underline{\alpha}} - \abs{n - s}^{1 - \underline{\alpha}}} \Di s
      \\
      &= \int_{-\infty}^{-n - 1} ((1 - s)^{1 - \underline{\alpha}} - (n - s)^{1 - \underline{\alpha}})^q \Di s
        + \int_{n + 1}^{\infty} ((s - n)^{1 - \underline{\alpha}} - (s - 1)^{1 - \underline{\alpha}})^q \Di s
      \\
      &= \int_{n + 1}^{\infty} (s^{1 - \underline{\alpha}} - (s + n)^{1 - \underline{\alpha}})^q \Di s
        + \int_{1}^{\infty} (s^{1 - \underline{\alpha}} - (s + n - 1)^{1 - \underline{\alpha}})^q \Di s
      \\
      &\leq n^{q(1 - \underline{\alpha}) + 1} \int_{n^{-1}}^{\infty} (s^{1 - \underline{\alpha}} - (s + 1 - \tfrac{1}{n})^{1 - \underline{\alpha}})^q \Di s
      \\
      &\leq n^{q(1 - \underline{\alpha}) + 1} \Bigl(\int_{n^{-1}}^{1} s^{q(1 - \underline{\alpha})} \Di s
        + \int_{1}^{\infty} s^{-\underline{\alpha} q} (1 - \tfrac{1}{n}) \Di s
        \Bigr)
        \leq C n^{q(1 - \underline{\alpha}) + 1}.
    \end{aligned}
  \end{equation}
  Combining \eqref{eq:geometric-bound}--\eqref{eq:x-0-1-outside-B-2}
  shows that if $q (d - \underline{\alpha}) \neq -1$, then
  \begin{equation}
    \label{eq:I_1-2}
    \begin{aligned}
      \MoveEqLeft \int_{-1}^{1} \biggl(\int_{[-n - 1, n + 1]\com} A_n^k(x, s)^q \Di s \biggr) \isp \nu(\di x)
      \\
      &\leq C n^{-q / 2} \int_{-1}^{1} \abs{x}^{q - 1 - \beta} \bigl( 1 + n^{q(1 - \underline{\alpha}) + 1} \bigr) \Di x
      \\
      &\leq C n^{1 - q / 2},
    \end{aligned}
  \end{equation}
  where the last inequality follows from $\underline{\alpha} >
  1$. Suppose $q(1 - \underline{\alpha}) = - 1$. Then
  \eqref{eq:kernel-ass} is satisfied with
  $\tilde{\underline{\alpha}} = \underline{\alpha} - \varepsilon$ and
  $q(d - \tilde{\underline{\alpha}}) \neq - 1$ for all
  $\varepsilon > 0$ sufficiently small. Therefore
  \eqref{eq:x-0-1-outside-B-2} holds for $\tilde{\underline{\alpha}}$
  and it is easily seen that in turn \eqref{eq:I_1-2} still holds when
  using $\tilde{\underline{\alpha}}$. So, combining \eqref{eq:I_1-1}
  and~\eqref{eq:I_1-2} proves that
  \begin{equation}
    \label{eq:I_1}
    I_1 \leq C n^{1 - q / 2} \quad \text{for all $n$ in $\bbN$.}
  \end{equation}

  \paragraph{Bounding $I_2$} We split $I_2$ as follows:
  \begin{align*}
    I_2 &\leq C \biggl(
          \int_{1}^{n^{\alpha}} x^{-1 - \beta} \Bigl(
          \int_{-n}^{n} A_n^i(x, s)^p \wedge A_n^j(x, s)^q \Di s
          \Bigr) \Di x
    \\
        &\quad+
          \int_{1}^{n^{\alpha}} x^{-1 - \beta} \Bigl(
          \int_{[-n, n]\com} A_n^i(x, s)^p \wedge A_n^j(x, s)^q \Di s
          \Bigr) \Di x
          \biggr)
    \\
        &\eqqcolon I_{2, 1} + I_{2, 2}.
  \end{align*}
  We consider first $I_{2, 1}$, but before splitting it further we
  split the sum defining $A_n^{\ell}$. For $s \in \bbR$ and $x > 1$ we write
  \begin{equation}
    \label{eq:A_n-split}
    \begin{aligned}
      n^{1 / 2} A_n^{\ell}(x, s) &= \sum_{i = 1}^{n} 1 \wedge \abs{x g_\ell(i - s)}
      \\
                                 &\leq \sup_{s \in \bbR} \# \Set{i \in \bbN \mid -1 \leq i - s \leq 1}
                                   + \sum_{i = 1}^{n} 1 \wedge \abs{x g_\ell(i - s)} \1_{[-1, 1]\com}(i - s)
      \\
                                 &\leq C x^{1 / \underline{\alpha}}
                                   + \sum_{i = 1}^{n} 1 \wedge x \abs{i - s}^{-\underline{\alpha}} \1_{[-1, 1]\com}(i - s)
      \\
                                 &\eqqcolon C x^{1 / \underline{\alpha}} + h_3(x, s).
    \end{aligned}
  \end{equation}
  We split $h_3$ additionally into two functions:
  \begin{equation}
    \label{eq:h_3}
    h_3(x, s) = \sum_{i = \ip{s} + 1}^{n} 1 \wedge x \abs{i - s}^{-\underline{\alpha}}
    + \sum_{i = 1}^{[s] - 1} 1 \wedge x \abs{i - s}^{-\underline{\alpha}}
    \eqqcolon h_{3, 1}(x, s) + h_{3, 2}(x, s)
  \end{equation}
  and note that $h_{3, 2}(x, s) = 0$ for $s \leq 1$. For $h_{3, 1}$
  we consider first the case where $s + x^{1/\underline{\alpha}} \leq n$,
  for which we have 
  \begin{equation}
    \label{eq:h_3-1}
    \begin{aligned}
      h_{3, 1}(x, s) &\leq \sum_{i = [s] + 1}^{[s + x^{1 / \underline{\alpha}}]} 1
                       + x \sum_{i = [s + x^{1 / \underline{\alpha}}] + 1}^n \abs{i - s}^{-\underline{\alpha}}
                       \leq x^{1 / \underline{\alpha}}
                       + C x ((x^{1 / \underline{\alpha}})^{1 - \underline{\alpha}}- (n - s)^{1 - \underline{\alpha}})
      \\
                     &\leq 2 x^{1 / \underline{\alpha}}.
    \end{aligned}
  \end{equation}
  In second case $s + x^{1 / \underline{\alpha}} > n$, $h_{3, 1}(x, s)$ is bounded as follows:
  \begin{equation}
    \label{eq:h_3-2}
    h_{3, 1}(x, s) \leq \sum_{i = [s] + 1}^{[s + x^{1 / \underline{\alpha}}]} 1 \leq 2 x^{1 / \underline{\alpha}},
  \end{equation}
  where we used that $x > 1$. The function $h_{3, 2}$ is split
  according to the minima, noting that
  $x \abs{i - s}^{-\underline{\alpha}} \leq 1$ if and only if
  $i \leq s - x^{1 / \underline{\alpha}}$, under the condition
  $i \leq s$. Hence,
  \begin{equation}
    \label{eq:h_3-3}
    \begin{aligned}
      h_{3, 2}(x, s) &= x \sum_{i = 1}^{[s - x^{1 / \underline{\alpha}}]} (s - i)^{-\underline{\alpha}}
                       + \sum_{i = [s - x^{1 / \underline{\alpha}}] + 1}^{[s] - 1} 1
      \\
                     &\leq C x ((s - [s - x^{1 / \underline{\alpha}}])^{1 - \underline{\alpha}} - (s - 1)^{1 - \underline{\alpha}})
                       + x^{1 / \underline{\alpha}}
      \\
                     &\leq 2 x^{1 / \underline{\alpha}}.
    \end{aligned}
  \end{equation}
  Combining \eqref{eq:h_3}--\eqref{eq:h_3-3} it follows from
  \eqref{eq:A_n-split} that
  \begin{equation}
    \label{eq:A_n-I_2}
    A_n^\ell(x, s) \leq C n^{-1 / 2} x^{1 / \underline{\alpha}}.
  \end{equation}
  Using \eqref{eq:A_n-I_2} we can proceed exactly as in
  equations~(4.28) and~(4.29) in \cite{BassBerr} to deduce that
  \begin{equation}
    \label{eq:I_21}
    I_{2, 1} \leq C
    \begin{cases}
      n^{1 - q / 2}, & \text{if $\underline{\alpha} \beta > q$,}
      \\
      n^{1 - q / 2} \log(n), & \text{if $\underline{\alpha} \beta = q$,}
      \\
      n^{(2 - \underline{\alpha} \beta) / 2}, & \text{if
        $2 < \underline{\alpha} \beta < 3$.}
    \end{cases}
  \end{equation}

  Before proceeding with $I_{2, 2}$ we first deduce the bound for
  $\abs{s} > x^{1 / \underline{\alpha}}$:
  \begin{equation}
    \label{eq:A_n-large-s}
    A_n^{\ell}(x, s) \leq x n^{-1 / 2} \sum_{t = 1}^{n} \abs{g_\ell(t - s)}
    \leq C x n^{1 / 2} \abs{s}^{-\underline{\alpha}}.
  \end{equation}
  Indeed, we start by observing that for all $n \in \bbN$,
  $t \in \Set{1, \ldots, n}$ and all such $s$ we have that
  $ \abs{t - s} \geq \abs[\big]{\abs{t} - \abs{s}} \geq
  x^{1/\underline{\alpha}} - n \geq 1.  $ Moreover, for any
  $n \in \bbN$, $\abs{s} > x^{1/\underline{\alpha}}$ and
  $t \in \Set{1, \ldots, n}$,
  $ \frac{\abs{t - s}}{\abs{s}} \geq \abs{1 - \frac{\abs{t}}{\abs{s}}}
  \geq \frac{1}{2}.  $ Hence, for all $n \in \bbN$, $s$ and
  $t \in \Set{1, \ldots, n}$, we have that
  $ \abs{t - s}^{-\underline{\alpha}} \leq 2^{\underline{\alpha}}
  \abs{s}^{-\underline{\alpha}}.  $ These considerations together with
  \eqref{eq:kernel-ass} lead directly to \eqref{eq:A_n-large-s}.

  We then split $I_{2, 2} = I_{2, 2, 1} + I_{2, 2, 2}$ according to whether or not
  $x n^{1 / 2} \abs{s}^{-\underline{\alpha}} < 1$, which happens for
  $x < n^{\underline{\alpha} - 1 / 2}$, whenever $\abs{s} > n$. Then,
  by \eqref{eq:A_n-large-s},
  \begin{equation}
    \label{eq:I_221}
    \begin{aligned}
      I_{2, 2, 1} &\coloneqq C \int_{1}^{n^{\underline{\alpha} - 1 / 2}} x^{-1 - \beta} \biggl( \int_{[-n, n]\com}
                    A_n^j(x, s)^p \wedge A_n^k(x, s)^q \Di s \biggr) \Di x
      \\
                  &\leq C n^{q / 2} \int_{1}^{n^{\underline{\alpha} - 1 / 2}} x^{q - 1 - \beta}
                    \biggl( \int_{[-n, n]\com} \abs{s}^{-\underline{\alpha} q} \Di s \Di x \biggr)
      \\
                  &= C n^{1 + q/2 - \underline{\alpha} q} (n^{(\underline{\alpha} - 1 / 2) (q - \beta)} - 1)
                    \leq C n^{1 + \beta/2 - \underline{\alpha} \beta}.
    \end{aligned}
  \end{equation}
  Before we split the second term we note that
  $x^{1 / \underline{\alpha}} n^{1 / (2 \underline{\alpha})} > n$ if
  and only if $x > n^{\underline{\alpha} - 1/2}$ and therefore we may
  split as follows:
  \begin{align*}
    I_{2, 2, 2} &\coloneqq C \int_{n^{\underline{\alpha} - 1/2}}^{n^{\underline{\alpha}}} x^{-1 - \beta} \biggl(\int_{n \leq \abs{s} \leq x^{1/\underline{\alpha}} n^{1/(2 \underline{\alpha})}} A_n^j(x, s)^p \wedge A_n^k(x, s)^q \Di s \biggr) \Di x
    \\
                &\qquad+ C \int_{n^{\underline{\alpha} - 1/2}}^{n^{\underline{\alpha}}} x^{-1 - \beta} \biggl( \int_{\abs{s} > x^{1/\underline{\alpha}} n^{1/(2 \underline{\alpha})}} A_n^j(x, s)^p \wedge A_n^k(x, s)^q \Di s \biggr) \Di x.
  \end{align*}
  The first term is bounded according to
  \eqref{eq:A_n-large-s} by
  \begin{equation}
    \label{eq:I_2221}
    \begin{aligned}
      \MoveEqLeft \int_{n^{\underline{\alpha} - 1/2}}^{n^{\underline{\alpha}}} x^{-1 - \beta} \biggl(\int_{n \leq \abs{s} \leq x^{1/\underline{\alpha}} n^{1/(2 \underline{\alpha})}} A_n^j(x, s)^p \wedge A_n^k(x, s)^q \Di s \biggr) \Di x
      \\
      &\leq n^{q / 2} \int_{n^{\underline{\alpha} - 1/2}}^{n^{\underline{\alpha}}} x^{p - 1 - \beta}
        \biggl(\int_{n \leq \abs{s} \leq x^{1/\underline{\alpha}} n^{1/(2 \underline{\alpha})}} \abs{s}^{-\underline{\alpha} p} \Di s \biggr) \Di x
      \\
      &\leq C n^{1 + p / 2 - \underline{\alpha} p} \int_{n^{\underline{\alpha} - 1/2}}^{n^{\underline{\alpha}}} x^{p - 1 - \beta} \Di x
      \\
      &\leq C
        \begin{cases}
          n^{1 + \beta / 2 - \underline{\alpha} \beta} \log(n), & \text{if
            $p = \beta$,}
          \\
          n^{1 + p / 2 - \underline{\alpha} \beta} + n^{1 + \beta / 2 -
            \underline{\alpha} \beta}, & \text{if $p \neq \beta$,}
        \end{cases}
    \end{aligned}
  \end{equation}
  where we used that $1 - \underline{\alpha} p < 0$ in the second
  inequality. The second term is bounded by
  \begin{equation}
    \label{eq:I_2222}
    \begin{aligned}
      \MoveEqLeft \int_{n^{\underline{\alpha} - 1/2}}^{n^{\underline{\alpha}}} x^{-1 - \beta} \biggl( \int_{\abs{s} > x^{1/\underline{\alpha}} n^{1/(2 \underline{\alpha})}} A_n^j(x, s)^p \wedge A_n^k(x, s)^q \Di s \biggr) \Di x
      \\
      &\leq C n^{p / 2} \int_{n^{\underline{\alpha} - 1/2}}^{n^{\underline{\alpha}}} x^{p - 1 - \beta}
        \int_{x^{1 / \underline{\alpha}} n^{1 / (2 \underline{\alpha})}}^{\infty} s^{-\underline{\alpha} p} \Di s \Di x
      \\
      &= C n^{1/(2 \underline{\alpha})} \int_{n^{\underline{\alpha} - 1/2}}^{n^{\underline{\alpha}}} x^{1 / \underline{\alpha} - 1 - \beta} \Di x
        \leq C n^{1 + \beta / 2 - \underline{\alpha} \beta},
    \end{aligned}
  \end{equation}
  where we used that $\underline{\alpha} \beta > 2 > 1$. Now, we
  gather the observations for the $I_{2, 2}$ term. Namely, from
  \eqref{eq:I_221}--\eqref{eq:I_2222} it follows that
  \begin{equation}
    \label{eq:I_22}
    I_{2, 2} \leq
    \begin{cases}
      n^{1 + \beta / 2 - \underline{\alpha} \beta} \log(n), & \text{if
        $p = \beta$,}
      \\
      n^{1 + p / 2 - \underline{\alpha} \beta} + n^{1 + \beta / 2 -
        \underline{\alpha} \beta}, & \text{if $p \neq \beta$.}
    \end{cases}
  \end{equation}
  We observe now that
  $1 + \beta / 2 - \underline{\alpha} \beta < (2 - \underline{\alpha}
  \beta) / 2$ since $\underline{\alpha} > 1$. Moreover, since
  $p \leq 2 < \underline{\alpha} \beta$ we have that
  $1 + p / 2 - \underline{\alpha} \beta < (2 - \underline{\alpha}
  \beta) / 2$. Using these observations for \eqref{eq:I_22} we see
  together with \eqref{eq:I_21} that
  \begin{align}
    I_2 = I_{2, 1} + I_{2, 2}
   & \leq
    \begin{cases}
      n^{1 - q / 2}, & \text{if $\underline{\alpha} \beta > q$,}
      \\
      n^{1 - q / 2} \log(n), & \text{if $\underline{\alpha} \beta = q$,}
      \\
      n^{(2 - \underline{\alpha} \beta) / 2}, & \text{if
        $2 < \underline{\alpha} \beta < 3$.}
    \end{cases}
    +
    \begin{cases}
      n^{1 + \beta / 2 - \underline{\alpha} \beta} \log(n), & \text{if
        $p = \beta$,}
      \\
      n^{1 + p / 2 - \underline{\alpha} \beta} + n^{1 + \beta / 2 -
        \underline{\alpha} \beta}, & \text{if $p \neq \beta$.}
    \end{cases}
                                     \nonumber
    \\
        &\leq
    \begin{cases}
      n^{1 - q / 2}, & \text{if $\underline{\alpha} \beta > q$,}
      \\
      n^{1 - q / 2} \log(n), & \text{if
        $\underline{\alpha} \beta = q$,}
      \\
      n^{(2 - \underline{\alpha} \beta) / 2}, & \text{if
        $2 < \underline{\alpha} \beta < 3$.}
    \end{cases}
                                                    \label{eq:I_2}
      \end{align}

  \paragraph{Bounding $I_3$} Finally, we deal with
  \begin{equation*}
    I_3 = 2 \int_{n^{\underline{\alpha}}}^{\infty} \int_{\bbR} A_n^j(x, s)^p \wedge A_n^k(x, s)^q \Di s \isp \nu(\di x).
  \end{equation*}
  We consider two cases for $s \in \bbR$. The first case is
  $\abs{s} > x^{1 / \underline{\alpha}}$ and we recall that in this case
  \eqref{eq:A_n-large-s} holds. Note that
  $x n^{1 / 2} \abs{s}^{-\underline{\alpha}} > 1$ if and only if
  $\abs{s} < x^{1 / \underline{\alpha}} n^{1 / (2
    \underline{\alpha})}$. When this is the case we obtain that
  \begin{align*}
    \int_{x^{1/\underline{\alpha}} < \abs{s} < x^{1/\underline{\alpha}} n^{1 / (2 \underline{\alpha})}} \abs{s}^{-\underline{\alpha} p} \Di s
    &= 2 \int_{x^{1 / \underline{\alpha}}}^{x^{1/\underline{\alpha}} n^{1 / (2 \underline{\alpha})}} s^{-\underline{\alpha} p} \Di s
    \\
    &\leq C x^{(1 - \underline{\alpha} p)/\underline{\alpha}} (1 + n^{(1 - \underline{\alpha} p)/(2 \underline{\alpha})}).
  \end{align*}
  In the other case we have that
  \begin{align*}
    \int_{\abs{s} > x^{1 / \underline{\alpha}} n^{1 / (2 \underline{\alpha})}} \abs{s}^{-\underline{\alpha} q} \Di s
    &= 2 \int_{x^{1 / \underline{\alpha}} n^{1 / (2 \underline{\alpha})}}^{\infty} s^{-\underline{\alpha} q} \Di s
    \\
    &= C x^{(1 - \underline{\alpha} q) / \underline{\alpha}} n^{(1 - \underline{\alpha} q) / (2 \underline{\alpha})},
  \end{align*}
  where we used that $\underline{\alpha} q > 1$.

  From these two we can conclude that
  \begin{equation}
    \label{eq:I_31}
      \begin{aligned}
    \MoveEqLeft \int_{n^{\underline{\alpha}}}^{\infty} \int_{\abs{s} > x^{1/\underline{\alpha}}} A_n^j(x, s)^p \wedge A_n^k(x, s)^q \Di s \isp \nu(\di x)
    \\
    &\leq C \int_{n^{\underline{\alpha}}}^{\infty} x^{-1 - \beta} \Bigl(x^p n^{p / 2} \int_{x^{1/\underline{\alpha}} < \abs{s} < x^{1/\underline{\alpha}} n^{1 / (2 \underline{\alpha})}} \abs{s}^{-\underline{\alpha} p} \Di s
    \\
    &\qquad + x^q n^{q / 2} \int_{\abs{s} > x^{1/\underline{\alpha}} n^{1 / (2 \underline{\alpha})}} \abs{s}^{- \underline{\alpha} q} \Di s\Bigr) \Di x
    \\
    &\leq C (n^{p / 2} + n^{p /2 + (1 - \underline{\alpha} p)/(2 \underline{\alpha})}) \int_{n^{\underline{\alpha}}}^{\infty} x^{-1 - \beta + p + (1 - \underline{\alpha} p)/\underline{\alpha}} \Di x
    \\
    &\qquad + C n^{q /2 + (1 - \underline{\alpha} q)/(2 \underline{\alpha})} \int_{n^{\underline{\alpha}}}^{\infty} x^{-1 - \beta + q + (1 - \underline{\alpha} q) / \underline{\alpha}} \Di x
    \\
    &= C \bigl(n^{1 + p / 2 - \underline{\alpha} \beta}
      + n^{1 + p / 2 + (1 - \underline{\alpha} p)/(2 \underline{\alpha}) - \underline{\alpha} \beta}
      + n^{1 + q / 2 + (1 - \underline{\alpha} q)/(2 \underline{\alpha}) - \underline{\alpha} \beta}
      \bigr)
    \\
    &\leq C (n^{1 + p / 2 - \underline{\alpha} \beta} + n^{1 + / (2 \underline{\alpha}) - \underline{\alpha} \beta}).
  \end{aligned}
  \end{equation}
  We consider now the second and last case
  $0 \leq \abs{s} \leq x^{1/\underline{\alpha}}$. Here, the
  trivial bound
  \begin{equation*}
    A_n^\ell(x, s) = n^{-1 / 2} \sum_{i = 1}^{n} 1 \wedge \abs{x g_\ell(i - s)}
    \leq n^{1 / 2}
  \end{equation*}
  will be sufficient. Indeed, by
  assumption~\eqref{eq:stable-levy-measure} we have that
  \begin{equation}
    \label{eq:I_32}
      \begin{aligned}
    \int_{n^{\underline{\alpha}}}^{\infty}
    \int_{\abs{s} \leq x^{1/\underline{\alpha}}} A_n^\ell(x, s)^p \Di s \Di x
    &\leq C n^{p / 2} \int_{n^{\underline{\alpha}}}^{\infty} x^{-1 - \beta} \biggl(
      \int_{0}^{x^{1 / \underline{\alpha}}} 1 \Di s \biggr) \Di x
    \\
    &= C n^{1 + p / 2 - \underline{\alpha} \beta},
  \end{aligned}
  \end{equation}
  where we used that $\underline{\alpha} \beta > 1$. Summarizing the
  inequalities \eqref{eq:I_31} and \eqref{eq:I_32} yields
  \begin{equation}
    \label{eq:I_3}
    I_3 \leq C (n^{1 + p / 2 - \underline{\alpha} \beta} + n^{1 + 1 / (2 \underline{\alpha}) - \underline{\alpha} \beta})
    \leq C n^{(2 - \underline{\alpha} \beta) / 2},
  \end{equation}
  where we used that
  $1 + 1 / (2 \underline{\alpha}) - \underline{\alpha} \beta < (2 -
  \underline{\alpha} \beta) / 2$ for the second term and the same
  considerations as in \eqref{eq:I_22} for the first term.

  Combining \eqref{eq:I_1}, \eqref{eq:I_2} and~\eqref{eq:I_3} we
  finally conclude \eqref{eq:A_n-ineq}.\qed
\end{proof}

\begin{proofof}
  According to Theorem~\ref{thm:wasserstein-dist} we have that for any
  $n\in\bbN$,
  \begin{equation*}
    d_3(V_n(X; f), Y)
    \leq \sum_{i, k = 1}^{d} (\gamma_{1}(F_i,F_k) + \gamma_{2}(F_i,F_k))
    + \gamma_3.
  \end{equation*}
  Using now Lemmas~\ref{lem:gamma1}, \ref{lem:gamma2}
  and~\ref{lem:gamma3} we see that
  \begin{align*}
    d_3(V_n(X; f), Y)
    &\leq C d^2 (m^3 n^{-1/2} + m^4 n^{-1/2}) + C d^4 m^4
      \begin{cases}
        n^{-1/2}, & \text{if $\underline{\alpha} \beta > 3$,}
        \\
        n^{-1/2} \log(n), & \text{if $\underline{\alpha} \beta = 3$,}
        \\
        n^{(2 - \underline{\alpha} \beta) / 2}, & \text{if
          $2 < \underline{\alpha} \beta < 3$,}
      \end{cases}
    \\
    &\leq C d^4 m^4
      \begin{cases}
        n^{-1/2}, & \text{if $\underline{\alpha} \beta > 3$,}
        \\
        n^{-1/2} \log(n), & \text{if $\underline{\alpha} \beta = 3$,}
        \\
        n^{(2 - \underline{\alpha} \beta) / 2}, & \text{if
          $2 < \underline{\alpha} \beta < 3$.}
      \end{cases}
  \end{align*}
  This completes the argument.\qed
\end{proofof}

\subsection*{Acknowledgement}

\noindent We would like to thank the two anonymous referees and the
associated editor for stimulating questions and remarks, which helped
us to significantly improve our paper.

\bibliography{References}

\end{document}